\documentclass[12pt, reqno]{amsart}
\usepackage{amsmath, amstext, amsbsy, amssymb, amscd}
\usepackage[mathscr]{eucal}

\newtheorem{theorem}{Theorem}[section]
\newtheorem{lemma}[theorem]{Lemma}
\newtheorem{proposition}[theorem]{Proposition}
\newtheorem{corollary}[theorem]{Corollary}
\theoremstyle{definition}

\theoremstyle{remark}
\newtheorem{remark}[theorem]{Remark}

\newtheorem{notation}[theorem]{Notation}

\newcommand{\C}{ \mathbb C}

\newcommand{\fa}{ \mathfrak a}
\newcommand{\fock}{ \mathbb H_X}

\newcommand{\Mbar}{ {\overline{\mathfrak M}}}
\newcommand{\Pee}{ \mathbb P}

\newcommand{\Supp}{ {\rm Supp}}

\newcommand{\vac}{|0\rangle}
\newcommand{\vir}{ {\rm vir}}
\newcommand{\w}{\tilde}
\newcommand{\W}{\widetilde}
\newcommand{\Xtwo}{ {X^{[2]}}}
\newcommand{\Xn}{ {X^{[n]}}}
\newcommand{\Z}{ \mathbb Z}

\def\beq{\begin{equation}}
\def\eeq{\end{equation}}

\numberwithin{equation}{section}

\begin{document}
\title[Gromov-Witten invariants of the Hilbert schemes]
{The Gromov-Witten invariants of the Hilbert schemes of points on surfaces with $p_g > 0$}

\author[Jianxun Hu]{Jianxun Hu$^1$}
\address{Department of Mathematics,  
             Sun Yat-Sen University, Guangzhou 510275, China}
\email{stsjxhu@mail.sysu.edu.cn}
\thanks{${}^1$Partially supported by NSFC grants 11228101 and 11371381}

\author[Wei-Ping Li]{Wei-Ping Li$^2$}
\address{Department of Mathematics, HKUST, Clear Water Bay, Kowloon, Hong Kong} 
\email{mawpli@ust.hk}
\thanks{${}^2$Partially supported by the grants FSGRF12SC10 and GRF602512}

\author[Zhenbo Qin]{Zhenbo Qin$^3$}
\address{Department of Mathematics, University of Missouri, Columbia, MO 65211, USA}
\email{qinz@missouri.edu}
\thanks{${}^3$Partially supported by NSFC grant 11228101
and an MU Summer Research Fellowship}

\date{\today}

\begin{abstract}
In this paper, we study the Gromov-Witten theory of the Hilbert schemes $\Xn$ of points 
on smooth projective surfaces $X$ with positive geometric genus $p_g$. 
Using cosection localization technique due to Y.~Kiem and J.~Li \cite{KL1, KL2},
we prove that if $X$ is a simply connected surface admitting 
a holomorphic differential two-form with irreducible zero divisor, 
then all the Gromov-Witten invariants of $\Xn$ defined via 
the moduli space $\Mbar_{g, r}(\Xn, \beta)$ vanish except possibly when 
$\beta = d_0 \beta_{K_X} - d \beta_n$ where $d$ is an integer, $d_0 \ge 0$
is a rational number, and $\beta_n$ and $\beta_{K_X}$ are defined in \eqref{BetaN} 
and \eqref{BetaC} respectively. When $n=2$, the exceptional cases can be further
reduced to the invariants: $\langle 1 \rangle_{0, \,\, \beta_{K_X} - d\beta_2}^{\Xtwo}$
with $K_X^2 = 1$ and $d \le 3$, and $\langle 1 \rangle_{1, d\beta_2}^{\Xtwo}$ with 
$d \ge 1$. We show that when $K_X^2 = 1$,
$$
\langle 1 \rangle_{0, \,\, \beta_{K_X} - 3 \beta_2}^{\Xtwo} 
= (-1)^{\chi(\mathcal O_X)}
$$ 
which is consistent with a well-known formula of Taubes \cite{Tau}. In addition, 
for an arbitrary smooth projective surface $X$ and $d \ge 1$, we verify that 
$$
\langle 1 \rangle_{1, d\beta_2}^{\Xtwo} = {K_X^2 \over 12d}.
$$
\end{abstract}

\maketitle

\section{\bf Introduction}

Cosection localization via holomorphic two-forms was introduced by Lee and Parker \cite{LP}
in symplectic geometry and by Kiem and J. Li \cite{KL1, KL2} in algebraic geometry.
It is a localization theorem on virtual cycles such as the virtual fundamental cycles
arising from Gromov-Witten theory. Using this technique, Kiem and J. Li \cite{KL1, KL2}
studied the Gromov-Witten theory of minimal surfaces of general type, and J. Li and 
the second author \cite{LL} computed the quantum boundary operator for the Hilbert
schemes of points on surfaces. Cosection localization also played a pivotal role in 
\cite{LQ2} determining the structure of genus-$0$ extremal Gromov-Witten invariants 
of these Hilbert schemes and verifying 
the Cohomological Crepant Resolution Conjecture for the Hilbert-Chow morphisms.

In this paper, we study the Gromov-Witten theory of the Hilbert schemes $\Xn$ of points 
on smooth projective surfaces $X$ with positive geometric genus 
$p_g = h^0(X, \mathcal O_X(K_X))$. 
Let $\Mbar_{g, r}(\Xn, \beta)$ be the moduli space of stable maps $\mu$ from 
genus-$g$ nodal curves $D$ with $r$-marked points and with $\mu_*[D] = \beta$ to $\Xn$.
Let $C$ be a smooth curve in $X$, and fix distinct points 
$x_1, \ldots, x_{n-1} \in X-C$. Define
\begin{eqnarray*}   
\beta_n &=& \left \{ \xi + x_2 + \ldots + x_{n-1} \in \Xn | \Supp(\xi) 
    = \{x_1\} \right \},   \\
\beta_C &=& \left \{ x + x_1 + \ldots + x_{n-1} \in \Xn | \, x \in C \right \}.
\end{eqnarray*}
By linearity, extend the notion $\beta_C$ to an arbitrary divisor $C$ 
(see \eqref{BetaCHei} for details).
Using cosection localization technique, we obtain the following vanishing result.

\begin{theorem}    \label{Intro-ThmVanish}
Let $X$ be a simply connected surface admitting a holomorphic differential two-form 
with irreducible zero divisor. If $\beta \ne d_0 \beta_{K_X} - d \beta_n$ for 
some integer $d$ and rational number $d_0 \ge 0$,
then all the Gromov-Witten invariants of $\Xn$ defined via the moduli space 
$\Mbar_{g, r}(\Xn, \beta)$ vanish.
\end{theorem}

When $n=2$ and $X$ is further assumed to be a minimal surface of general type, 
the possible non-vanishing Gromov-Witten invariants
$\langle \alpha_1, \ldots, \alpha_r \rangle_{g, \beta}^{\Xtwo}$ (see \eqref{def-GW} 
for the precise definition) can be reduced to the $1$-point invariants 
calculated in \cite{LQ1} and the following two types of invariants: 
\begin{enumerate}
\item[{\rm (i)}] $\langle 1 \rangle_{1, d\beta_2}^{\Xtwo}$ with $d \ge 1$;

\item[{\rm (ii)}] $\langle 1 \rangle_{0, \,\, \beta_{K_X} - d\beta_2}^{\Xtwo}$ with $K_X^2 = 1$, 
$1 \le p_g \le 2$ and $d \le 3$.
\end{enumerate}
These two types of Gromov-Witten invariants are investigated via detailed analyses 
of the corresponding virtual fundamental cycles. 

\begin{theorem}   \label{Intro-theorem_ii}
Let $d \ge 1$. Let $X$ be a smooth projective surface. Then, 
$$
\langle 1 \rangle_{1, d\beta_2}^{\Xtwo} = {K_X^2 \over 12d}.
$$
\end{theorem}

It follows that if $X$ is a simply connected surface admitting a holomorphic differential 
two-form with irreducible zero divisor and satisfying $K_X^2 > 1$,
then all the Gromov-Witten invariants (without descendant insertions)
of $\Xtwo$ can be determined; moreover, the quantum cohomology of $\Xtwo$
coincides with its quantum corrected cohomology \cite{LQ1, LQ2}.

\begin{theorem}   \label{Intro-theorem_iii}
Let $X$ be a simply connected minimal surface of general type with 
$K_X^2 = 1$ and $1 \le p_g \le 2$ such that every member in $|K_X|$ is smooth. Then,
\begin{enumerate}
\item[{\rm (i)}] $\Mbar_{0, 0}(\Xtwo, \beta_{K_X} - 3\beta_2) \cong |K_X| \cong \Pee^{p_g-1}$;

\item[{\rm (ii)}] $\langle 1 \rangle_{0, \,\, \beta_{K_X} - 3 \beta_2}^{\Xtwo} 
= (-1)^{\chi(\mathcal O_X)}$.
\end{enumerate}
\end{theorem}

We remark that our formula in Theorem~\ref{Intro-theorem_iii}~(ii) is consistent with 
$$
\langle 1 \rangle_{K_X^2+1, \,\, K_X}^X = (-1)^{\chi(\mathcal O_X)}
$$
which is a well-known formula of Taubes \cite{Tau}
obtained via an interplay between Seiberg-Witten theory and Gromov-Witten theory. 

This paper is organized as follows. In \S 2, we briefly review Gromov-Witten theory.
In \S 3, Theorem~\ref{Intro-ThmVanish} is proved. In \S 4, we compute some intersection
numbers on certain moduli spaces of genus-$1$ stable maps. In \S 5, 
we study the homology classes of curves in Hilbert schemes of
points on surfaces. In \S 6, using the results from the previous two sections, 
we verify Theorem~\ref{Intro-theorem_ii} and Theorem~\ref{Intro-theorem_iii}. 


\medskip\noindent
{\bf Acknowledgment}: We thank Professors Dan Edidin and Jun Li for valuable helps. 
The third author also thanks HKUST and Sun Yat-Sen University for their hospitality 
and financial support during his visits in Summers 2013 and 2014.
\section{Stable maps and Gromov-Witten invariants}   
\label{Stable}

In this section, we will briefly review the notions of stable maps 
and Gromov-Witten invariants. We will also recall a result of Behrend from \cite{Beh}.

Let $Y$ be a smooth projective variety.
An $r$-pointed stable map to $Y$ consists of
a complete nodal curve $D$ with $r$ distinct ordered smooth points
$p_1, \ldots, p_r$ and a morphism $\mu: D \to Y$ such that
the data $(\mu, D, p_1, \ldots, p_r)$ has only finitely many automorphisms.
In this case, the stable map is denoted by
$[\mu: (D; p_1, \ldots, p_r) \to Y]$.
For a fixed homology class $\beta \in H_2(Y, \mathbb Z)$,
let $\overline {\frak M}_{g, r}(Y, \beta)$ be the coarse moduli space
parameterizing all the stable maps $[\mu: (D; p_1, \ldots, p_r) \to Y]$
such that $\mu_*[D] = \beta$ and the arithmetic genus of $D$ is $g$.
Then, we have the $i$-th evaluation map:
\begin{eqnarray}\label{evk}
{\rm ev}_i \colon \overline {\frak M}_{g, r}(Y, \beta) \to Y  
\end{eqnarray}
defined by ${\rm ev}_i([\mu: (D; p_1, \ldots, p_r) \to Y]) = 
\mu(p_i)$. It is known \cite{LT1, LT2, BF} that 
the coarse moduli space $\overline {\frak M}_{g, r}(Y, \beta)$ is projective and
has a virtual fundamental class
$[\overline {\frak M}_{g, r}(Y, \beta)]^{\text{vir}} \in
A_{\frak d}(\overline {\frak M}_{g, r}(Y, \beta))$ where
\begin{eqnarray}\label{expected-dim}
\frak d = -(K_Y \cdot \beta) + (\dim (Y) - 3)(1-g) + r 
\end{eqnarray}
is the expected complex dimension of 
$\overline {\frak M}_{g, r}(Y, \beta)$, 
and $A_{\frak d}(\overline {\frak M}_{g, r}(Y, \beta))$
is the Chow group of $\frak d$-dimensional cycles in 
the moduli space $\overline {\frak M}_{g, r}(Y, \beta)$. 

The Gromov-Witten invariants are defined by using
the virtual fundamental class
$[\overline {\frak M}_{g, r}(Y, \beta)]^{\text{vir}}$.
Recall that an element
$\alpha \in H^*(Y, \mathbb C) {\buildrel\text{def}\over=}
\bigoplus_{j=0}^{2 \dim_{\mathbb C}(Y)} H^j(Y, \mathbb C)$ is 
{\it homogeneous} if $\alpha \in H^j(Y, \mathbb C)$ for some $j$; 
in this case, we take $|\alpha| = j$. 
Let $\alpha_1, \ldots, \alpha_r \in H^*(Y, \mathbb C)$
such that every $\alpha_i$ is homogeneous and
\begin{eqnarray}\label{homo-deg}
\sum_{i=1}^r |\alpha_i| = 2 {\frak d}.  
\end{eqnarray}
Then, we have the $r$-point Gromov-Witten invariant defined by:
\begin{eqnarray}\label{def-GW}
\langle \alpha_1, \ldots, \alpha_r \rangle_{g, \beta}^Y \,\,
= \int_{[\overline {\frak M}_{g, r}(Y, \beta)]^{\text{vir}}}
{\rm ev}_1^*(\alpha_1) \otimes \ldots \otimes {\rm ev}_r^*(\alpha_r).  
\end{eqnarray}

Next, we recall that {\it the excess dimension} is the difference between 
the dimension of $\overline {\frak M}_{g, r}(Y, \beta)$ and
the expected dimension $\frak d$ in (\ref{expected-dim}).
Let $T_Y$ stand for the tangent sheaf of $Y$. For $0 \le i < r$, we shall use 
\begin{eqnarray}\label{r-to-i}
f_{r, i}: \overline {\frak M}_{g, r}(Y, \beta) \to
\overline {\frak M}_{g, i}(Y, \beta)  
\end{eqnarray}
to stand for the forgetful map
obtained by forgetting the last $(r-i)$ marked points
and contracting all the unstable components.
It is known that $f_{r, i}$ is flat when $\beta \ne 0$ and $0 \le i < r$.
The following can be found in \cite{Beh}.

\begin{proposition}   \label{virtual-prop} 
Let $\beta \in H_2(Y, \mathbb Z)$ and $\beta \ne 0$.
Let $e$ be the excess dimension of the moduli space $\overline {\frak M}_{g, r}(Y, \beta)$. 
If $R^1(f_{r+1, r})_*({\rm ev}_{r+1})^*T_Y$ is a rank-$e$ locally free sheaf
over $\overline {\frak M}_{g, r}(Y, \beta)$, then $\overline {\frak M}_{g, r}(Y, \beta)$ 
is smooth (as a stack) of dimension 
\begin{eqnarray}   \label{virtual-prop.0}
\frak d + e = -(K_Y \cdot \beta) + (\dim (Y) - 3)(1-g) + r + e,
\end{eqnarray}
and $[\overline {\frak M}_{g, r}(Y, \beta)]^{\text{\rm vir}} 
= c_e \big (R^1(f_{r+1, r})_*({\rm ev}_{r+1})^*T_Y \big ) 
\cap [\overline {\frak M}_{g, r}(Y, \beta)/\overline {\frak M}_{g, r}]$.
\end{proposition}
\section{\bf Vanishing of Gromov-Witten invariants}
\label{sect_Vanishing}

In this section, we will recall some basic notations regarding the Hilbert schemes of 
points on surfaces, and prove Theorem~\ref{Intro-ThmVanish}.

Let $X$ be a smooth projective complex surface, and $\Xn$ be the Hilbert scheme of points in $X$.
An element in $\Xn$ is represented by a length-$n$ $0$-dimensional closed subscheme $\xi$ of $X$. 
For $\xi \in \Xn$, let $I_{\xi}$ and  $\mathcal O_\xi$ be the corresponding sheaf of ideals and 
structure sheaf respectively. It is known from \cite{Fog1, Iar} that $\Xn$ is 
a smooth irreducible variety of dimension $2n$. 
The universal codimension-$2$ subscheme is 
\beq    \label{UnivZn}
\mathcal Z_n=\{(\xi, x) \in \Xn \times X \, |\, x\in \Supp{(\xi)}\} 
\quad \subset \Xn \times X.
\eeq
The boundary of $\Xn$ is defined to be the subset 
$$
B_n = \left \{ \xi \in \Xn|\, |\Supp{(\xi)}| < n \right \}.
$$
Let $C$ be a real-surface in $X$, and fix distinct points $x_1, \ldots, x_{n-1} \in X$ 
which are not contained in $C$. Define the subsets
\begin{eqnarray}   
\beta_n &=& \left \{ \xi + x_2 + \ldots + x_{n-1} \in \Xn | \Supp(\xi) = \{x_1\} \right \},   
        \label{BetaN}   \\
\beta_C &=& \left \{ x + x_1 + \ldots + x_{n-1} \in \Xn | \, x \in C \right \},   
        \label{BetaC}   \\
D_C &=& \left \{ \xi \in \Xn | \, \Supp(\xi) \cap C \ne \emptyset \right \}.  
        \label{DC}
\end{eqnarray}
Note that $\beta_C$ (respectively, $D_C$) is a curve (respectively, a divisor) in $\Xn$ when 
$C$ is a smooth algebraic curve in $X$. We extend the notions $\beta_C$ and $D_C$ to 
all the divisors $C$ in $X$ by linearality. For a subset $Y \subset X$, define
$$
M_n(Y) = \{ \xi \in \Xn| \, \Supp(\xi) \text{ is a point in $Y$}\}.
$$

Nakajima \cite{Nak} and Grojnowski \cite {Gro} geometrically constructed a Heisenberg algebra 
action on the cohomology of the Hilbert schemes $\Xn$. Denote the Heisenberg operators by 
$\fa_m(\alpha)$ where $m \in \Z$ and $\alpha \in H^*(X, \C)$. Put
$$
\fock = \bigoplus_{n=0}^{+ \infty} H^*(\Xn, \C).
$$
Then the space $\fock$ is an irreducible representation of the Heisenberg algebra 
generated by the operators $\fa_m(\alpha)$ with the highest weight vector being
$\vac = 1 \in H^*(X^{[0]}, \C) = \C$. It follows that the $n$-th component 
$H^*(\Xn, \C)$ in $\fock$ is linearly spanned by the {\it Heisenberg monomial classes}:
$$
\mathfrak a_{-n_1}(\alpha_1) \cdots \mathfrak a_{-n_k}(\alpha_k) \vac
$$
where $k \ge 0$, $n_1, \ldots, n_k > 0$, and $n_1 + \ldots + n_k = n$.  We have
\begin{eqnarray}   
\beta_n &=& \fa_{-2}(x) \fa_{-1}(x)^{n-2} \vac,   \label{BetaNHei}   \\
\beta_C &=& \fa_{-1}(C) \fa_{-1}(x)^{n-1} \vac,    \label{BetaCHei}   \\
B_n &=& {1 \over (n-2)!} \fa_{-1}(1_X)^{n-2} \fa_{-2}(1_X) \vac,    \label{BNHei}   \\
D_C &=& {1 \over (n-1)!} \fa_{-1}(1_X)^{n-1} \fa_{-1}(C) \vac         \label{DCHei}
\end{eqnarray}
where $x$ and $1_X$ denote the cohomology classes corresponding to a point $x \in X$ and 
the surface $X$ respectively. By abusing notations, we also use $C$ to denote 
the cohomology class corresponding to the real-surface $C$.

Assume that the surface $X$ admits a non-trivial holomorphic differential two-form 
$\theta \in H^0(X, \Omega_X^2) = H^0(X, \mathcal O_X(K_X))$. 
By the results of Beauville in \cite{Bea1, Bea2},
$\theta$ induces a holomorphic two-form $\theta^{[n]}$ of the Hilbert scheme $\Xn$ 
which can also be regarded as a map $\theta^{[n]}: T_\Xn \to \Omega_\Xn$. 
For simplicity, put
$$
\Mbar = \Mbar_{g, r}(\Xn, \beta).
$$
Define the degeneracy locus $\Mbar(\theta)$ to be the subset of $\Mbar$ consisting of 
all the stable maps $u: \Gamma \to \Xn$ such that the composite
\beq   \label{thetaNull}
u^*(\theta^{[n]}) \circ du: \quad T_{\Gamma_{\rm reg}} \to 
u^*T_{\Xn}|_{\Gamma_{\rm reg}} \to u^*\Omega_{\Xn}|_{\Gamma_{\rm reg}}
\eeq
is trivial over the regular locus $\Gamma_{\rm reg}$ of $\Gamma$.
By the results of Kiem-Li \cite{KL1, KL2}, $\theta^{[n]}$ defines 
a regular cosection of the obstruction sheaf of $\Mbar$:
\beq    \label{LL3.1}
\eta : \mathcal Ob_{\Mbar} \longrightarrow \mathcal O_{\Mbar}
\eeq
where $\mathcal Ob_{\Mbar}$ is the obstruction sheaf and $\mathcal O_{\Mbar}$ is 
the structure sheaf of $\Mbar$. Moreover, the cosection $\eta$ is surjective away from 
the degeneracy locus $\Mbar(\theta)$, and there exists a localized virtual cycle 
$[\Mbar]^{\rm vir}_{\rm loc} \in A_*(\Mbar(\theta))$ such that
\beq   \label{KL1.1}
[\Mbar]^{\rm vir} = \iota_*[\Mbar]^{\rm vir}_{\rm loc} \in A_*(\Mbar)
\eeq
where $\iota: \Mbar(\theta) \to \Mbar$ stands for the inclusion map.

\begin{lemma}   \label{LmaNull}
Let $C_0$ be the zero divisor of $\theta$. Let $u: \Gamma \to \Xn$ be a stable map 
in $\Mbar(\theta)$, and let $\Gamma_0$ be an irreducible component of $\Gamma$ 
with non-constant restriction $u|_{\Gamma_0}$. Then there exists $\xi_1 \in X^{[n_0]}$ 
for some $n_0$ such that $\Supp(\xi_1) \cap C_0 = \emptyset$ and 
\beq    \label{LmaNull.0}
u(\Gamma_0) \subset \xi_1 + \{\xi_2| \Supp(\xi_2) \subset C_0 \}.
\eeq
\end{lemma}
\begin{proof}
For notational convenience, we assume that $\Gamma = \Gamma_0$ is irreducible. 
Then there exist a nonempty open subset $O \subset \Gamma$ and an integer 
$n_0 \ge 0$ such that $O$ is smooth and for every element $p \in O$, 
the image $u(p)$ is of the form
\beq    \label{LmaNull.1}
u(p) = \xi_1(p) + \xi_2(p)
\eeq
where $\xi_1(p) \in X^{[n_0]}$ with $\Supp(\xi_1(p)) \cap C_0 = \emptyset$ and
$\Supp(\xi_2(p)) \subset C_0$. This induces a decomposition $u|_O = (u_1, u_2)$
where the morphisms $u_1: O \to X^{[n_0]}$ and $u_2: O \to X^{[n-n_0]}$ are 
defined by sending $p \in O$ to $\xi_1(p)$ and $\xi_2(p)$ respectively.
Since \eqref{thetaNull} is trivial over the regular locus $\Gamma_{\rm reg}$ of $\Gamma$,
the composite
\beq   \label{LmaNull.2}
u_1^*(\theta^{[n_0]}) \circ du_1: \quad T_O \to  u_1^*T_{X^{[n_0]}}|_O 
\to u_1^*\Omega_{X^{[n_0]}}|_O
\eeq
is trivial. Note that the holomorphic two-form $\theta^{[n_0]}$ on $X^{[n_0]}$ is 
non-degenerate at $\xi_1(p), p \in O$ since $\Supp(\xi_1(p)) \cap C_0 = \emptyset$.
Thus, $du_1 = 0$ and $u_1$ is a constant morphism. 
Setting $\xi_1 = \xi_1(p) = u_1(p), p \in O$ proves the lemma.
\end{proof}

In the rest of the paper, we will assume that $X$ is simply connected. Then,
\begin{eqnarray}   \label{PicardH1=0}
{\rm Pic}(\Xn) \cong {\rm Pic}(X) \oplus \Z \cdot (B_n/2)
\end{eqnarray}
by \cite{Fog2}. Under this isomorphism, the divisor $D_C \in {\rm Pic}(\Xn)$ 
corresponds to $C \in {\rm Pic}(X)$. Let $\{\alpha_1, \ldots, \alpha_s\}$ 
be a linear basis of $H^2(X, \C)$.
Then, 
\beq   \label{BasisH^2}
\{D_{\alpha_1}, \ldots, D_{\alpha_s}, B_n\}
\eeq
is a linear basis of $H^2(\Xn, \C)$. Represent $\alpha_1, \ldots, \alpha_s$ by real-surfaces 
$C_1, \ldots, C_s \subset X$ respectively. Then a linear basis of $H_2(\Xn, \C)$ is given by
\beq   \label{BasisH_2}
\{\beta_{C_1}, \ldots, \beta_{C_s}, \beta_n\}.
\eeq

\begin{lemma}   \label{MThetaNonEmpty}
Let the surface $X$ be simply connected. Assume that the zero divisor $C_0$ of $\theta$ 
is irreducible. If the subset $\Mbar(\theta)$ of $\Mbar = \Mbar_{g, r}(\Xn, \beta)$ is nonempty,
then $\beta = d_0 \beta_{C_0} - d \beta_n$ for some integer $d$ and some rational number 
$d_0 \ge 0$. Moreover, if $C_0$ is also reduced, then $d_0$ is an non-negative integer.
\end{lemma}
\begin{proof}
Let $u: \Gamma \to \Xn$ be a stable map in $\Mbar(\theta)$. Restricting $u$ 
to the irreducible components of $\Gamma$ if necessary, we may assume that
$\Gamma$ is irreducible. By Lemma~\ref{LmaNull}, there exists $\xi_1 \in X^{[n_0]}$ 
for some $n_0$ such that $\Supp(\xi_1) \cap C_0 = \emptyset$ and 
\beq    \label{MThetaNonEmpty.1}
u(\Gamma) \subset \xi_1 + \{\xi_2| \Supp(\xi_2) \subset C_0 \}.
\eeq
We may further assume that $n_0 = 0$ and $C_0$ is reduced. Then for every $p \in \Gamma$,
$$
\Supp(u(p)) \subset C_0.
$$

Let $C$ be a real-surface in $X$. Assume that $C_0$ and $C$ intersect transversally 
at $x_{1, 1}, \ldots, x_{1, s}, x_{2, 1}, \ldots, x_{2, t} \in (C_0)_{\rm reg}$ such that 
each $x_{1, i}$ (respectively, $x_{2,i}$) contributes $1$ (respectively, $-1$) 
to the intersection number $C_0 \cdot C$. 
So $s - t = C_0 \cdot C$. Since $\Supp(u(p)) \subset C_0$ and $C_0$ 
is irreducible and reduced, there exists an integer $d_0'$ such that $d_0'$ is independent
of $C$ and that each $x_{1, i}$ (respectively, $x_{2,i}$) contributes $d_0'$ 
(respectively, $-d_0'$) to the intersection number $u(\Gamma) \cdot D_C$. Thus, 
$$
u(\Gamma) \cdot D_C = sd_0' - td_0' = (s-t)d_0' = d_0' (C_0 \cdot C).
$$
In view of the bases \eqref{BasisH_2} and \eqref{BasisH^2}, 
$u(\Gamma) = d_0' \beta_{C_0} - d' \beta_n$ for some integer $d'$.
Choosing $C$ to be a very ample curve, we see that $d_0' \ge 0$.
Finally, since $\beta = \deg(u) \cdot u(\Gamma)$, we obtain
$\beta = d_0 \beta_{C_0} - d \beta_n$ for some integers $d_0 \ge 0$ and $d$.
\end{proof}

\begin{theorem}    \label{ThmVanish}
Let $X$ be a simply connected surface admitting a holomorphic differential two-form 
with irreducible zero divisor. If $\beta \ne d_0 \beta_{K_X} - d \beta_n$ for 
some integer $d$ and rational number $d_0 \ge 0$,
then all the Gromov-Witten invariants of $\Xn$ defined via the moduli space 
$\Mbar_{g, r}(\Xn, \beta)$ vanish.
\end{theorem}
\begin{proof}
Let $\theta \in H^0(X, \Omega_X^2) = H^0(X, \mathcal O_X(K_X))$ be the holomorphic differential 
two-form whose zero divisor $C_0$ is irreducible.  
By Lemma~\ref{MThetaNonEmpty}, we have $\Mbar(\theta) = \emptyset$.
It follows from \eqref{KL1.1} that $[\Mbar_{g, r}(\Xn, \beta)]^{\rm vir} = 0$. 
Therefore, all the Gromov-Witten invariants defined via the moduli space 
$\Mbar_{g, r}(\Xn, \beta)$ vanish.
\end{proof}

\begin{remark} \label{RmkThmVanish}
From the proof of Lemma~\ref{MThetaNonEmpty}, we see that if $K_X = C_0 = mC_0'$
for some irreducible and reduced curve $C_0'$, then the rational number $d_0$
in Theorem~\ref{ThmVanish} is of the form $d_0'/m$ for some integer $d_0' \ge 0$.
\end{remark}

Recall that $K_\Xn = D_{K_X}$. Thus, if $\beta = d_0 \beta_{K_X} - d \beta_n$ for 
some rational number $d_0 \ge 0$ and integer $d$, then the expected dimension of 
$\Mbar_{g, r}(\Xn, \beta)$ is
\begin{eqnarray}    \label{ExpDim}
      \mathfrak d 
&=&-K_\Xn \cdot \beta + (\dim \Xn - 3)(1-g)+ r    \nonumber    \\
&=&-d_0K_X^2 + (2n - 3)(1-g)+ r.
\end{eqnarray}

Our first corollary deals with the case when $X$ is an elliptic surface.

\begin{corollary}  \label{EllipticCor}
Let $X$ be a simply connected (minimal) elliptic surface without multiple fibers 
and with positive geometric genus. 
Let $n \ge 2$ and $\beta \ne 0$. Then all the Gromov-Witten invariants without descendant 
insertions defined via the moduli space $\Mbar_{g, r}(\Xn, \beta)$ vanish, 
except possibly when $0 \le g \le 1$ and $\beta = d_0 \beta_{K_X} - d \beta_n$ for 
some integer $d$ and rational number $d_0 \ge 0$.
\end{corollary}
\begin{proof}
Since $X$ is a simply connected elliptic surface without multiple fibers,
$K_X = (p_g-1)f$ where $p_g \ge 1$ is the geometric genus of $X$ and $f$ denotes 
a smooth fiber of the elliptic fibration. By Theorem~\ref{ThmVanish}, 
it remains to consider the case when $\beta = d_0 \beta_{K_X} - d \beta_n$ for 
some integer $d$ and rational number $d_0 \ge 0$. By \eqref{ExpDim} and $K_X^2 = 0$, 
the expected dimension of the moduli space $\Mbar_{g, r}(\Xn, \beta)$ is equal to 
$\mathfrak d = (2n - 3)(1-g)+ r$. By the Fundamental Class Axiom, all the Gromov-Witten 
invariants without descendant insertions are equal to zero if $g \ge 2$.
\end{proof}

Our second corollary concentrates on the case when $X$ is of general type.

\begin{corollary}  \label{GenTypeCor}
Let $X$ be a simply connected minimal surface of general type admitting a holomorphic 
differential two-form with irreducible zero divisor. 
Let $n \ge 2$ and $\beta \ne 0$. 
Then all the Gromov-Witten invariants without descendant insertions defined via 
$\Mbar_{g, r}(\Xn, \beta)$ vanish, except possibly in the following cases
\begin{enumerate}
\item[{\rm (i)}] $g = 0$ and $\beta = d \beta_n$ for some integer $d > 0$;

\item[{\rm (ii)}] $g = 1$ and $\beta = d \beta_n$ for some integer $d > 0$;

\item[{\rm (iii)}] $g = 0$ and $\beta = d_0 \beta_{K_X} - d \beta_n$ for some integer $d$ and rational number $d_0 > 0$.
\end{enumerate}
\end{corollary}
\begin{proof}
In view of Theorem~\ref{ThmVanish}, it remains to consider the case when
$\beta = d_0 \beta_{K_X} - d \beta_n$ for some integer $d$ and rational number $d_0 \ge 0$.

When $d_0 = 0$ and $\beta = d \beta_n$ with $d > 0$, we see from \eqref{ExpDim} that
the expected dimension of the moduli space $\Mbar_{g, r}(\Xn, \beta)$ is equal to 
$$
\mathfrak d = (2n - 3)(1-g)+ r.
$$
If $g \ge 2$, then all the Gromov-Witten invariants without descendant insertions defined via 
$\Mbar_{g, r}(\Xn, \beta)$ vanish by the Fundamental Class Axiom.

Next, assume that $d_0 > 0$. Since $K_X^2 \ge 1$, we see from \eqref{ExpDim} that
$$
\mathfrak d < (2n - 3)(1-g)+ r.
$$
By the Fundamental Class Axiom, all the Gromov-Witten invariants without descendant insertions 
vanish except possibly in the case when $g = 0$.
\end{proof}
\section{\bf Intersection numbers on some moduli space of genus-$1$ stable maps}
\label{sect_ProjV}

In this section, we will compute certain intersection numbers on the moduli space of 
genus-$1$ stable maps to $\Pee(V)$ where $V$ is a rank-$2$ vector bundle over
a smooth projective curve $C$. The results will be used in Subsection~\ref{subsect_n=2Case2}.

\begin{notation}  \label{Notation}
Let $V$ be a rank-$2$ bundle over a smooth projective variety $B$.
\begin{enumerate}
\item[{\rm (i)}] $f$ denotes a fiber of the ruling $\pi: \Pee(V) \to B$ or 
its cohomology class.

\item[{\rm (ii)}] $\mathcal H = (f_{1,0})_*\omega$ is the rank-$1$ Hodge bundle
over $\Mbar_{1, 0}(\Pee(V), df)$ where $\omega$ is the relative 
dualizing sheaf for $f_{1,0}: \Mbar_{1, 1}(\Pee(V), df) \to \Mbar_{1, 0}(\Pee(V), df)$. 

\item[{\rm (iii)}] $\lambda = c_1(\mathcal H)$.
\end{enumerate}
\end{notation}

Let $d \ge 1$. If $u = [\mu: D \to \Pee(V)] \in \Mbar_{1, 0}(\Pee(V), df)$,
then $\mu(D)$ is a fiber of the ruling $\pi: \Pee(V) \to B$. Therefore, 
there exists a natural morphism 
\beq   \label{phiToB}
\phi: \quad \Mbar_{1, 0}(\Pee(V), df) \to B
\eeq
whose fiber over $b \in B$ is 
$\Mbar_{1, 0}\big (\pi^{-1}(b), d[\pi^{-1}(b)] \big ) \cong \Mbar_{1, 0}(\Pee^1, d[\Pee^1])$.
So the moduli space $\Mbar_{1, 0}(\Pee(V), df)$ is smooth (as a stack) with dimension
$$
\dim \Mbar_{1,0}(\Pee^1, d[\Pee^1]) + \dim (B) = 2d + \dim (B).
$$
By \eqref{expected-dim}, the expected dimension of $\Mbar_{1, 0}(\Pee(V), df)$ is $2d$.
Since $d \ge 1$, the sheaf $R^1 (f_{1,0})_*{\rm ev}_1^*\mathcal O_{\Pee(V)}(-2)$ on 
$\Mbar_{1, 0}(\Pee(V), df)$ is locally free of rank-$2d$. In addition, 
\beq   \label{Mumford}
\lambda^2 = 0
\eeq
according to Mumford's theorem in \cite{Mum} regarding the Chern character of 
the Hodge bundles and the proof of Proposition~1 in \cite{FP}.

\begin{lemma}   \label{deformation}
Let $d \ge 1$. Let $V$ be a rank-$2$ bundle over $B_0 \times C$ where $B_0$ and $C$
are smooth projective curves. Let $V_b = V|_{\{b\} \times C}$ for $b \in B_0$. Then, 
\begin{eqnarray}     \label{deformation.0}
\int_{[\Mbar_{1, 0}(\Pee(V_b), df)]} \lambda \cdot
c_{2d} \big ( R^1 (f_{1,0})_*{\rm ev}_1^*\mathcal O_{\Pee(V_b)}(-2) \big )     
\end{eqnarray}
is independent of the points $b \in B_0$.
\end{lemma}
\begin{proof}
This follows from the observation that \eqref{deformation.0} is equal to
$$
\int_{[\Mbar_{1, 0}(\Pee(V), df)]} \phi^*[\{b\} \times C] \cdot \lambda \cdot
c_{2d} \big ( R^1 (f_{1,0})_*{\rm ev}_1^*\mathcal O_{\Pee(V)}(-2) \big )
$$
where the morphism $\phi: \Mbar_{1, 0}(\Pee(V), df) \to B_0 \times C$ is from \eqref{phiToB}.
\end{proof}

Formula \eqref{0K3.01} below is probably well-known, but we could not find a reference.

\begin{lemma}   \label{0K3}
Let $d$ be a positive integer. Then, we have
\begin{eqnarray} 
\int_{[\Mbar_{1, 0}(\Pee^1, d[\Pee^1])]} 
   c_{2d} \big ( R^1 (f_{1,0})_*{\rm ev}_1^*\mathcal O_{\Pee^1}(-2) \big ) 
&=&0,       \label{0K3.01}  \\
\int_{[\Mbar_{1, 0}(\Pee^1, d[\Pee^1])]} \lambda \cdot
   c_{2d-1} \big ( R^1 (f_{1,0})_*{\rm ev}_1^*\mathcal O_{\Pee^1}(-2) \big ) 
&=&-{1 \over 12d}.     \label{0K3.02}  
\end{eqnarray}
\end{lemma}
\begin{proof}
We begin with the proof of \eqref{0K3.01}. 
Choose a K3 surface $S$ which contains a smooth rational curve $C$. 
Then, $C^2 = -2$, $T_S|_C = \mathcal O_C(2) \oplus \mathcal O_C(-2)$,
and $dC$ is the only element in the complete linear system $|dC|$. So we have
\begin{eqnarray*} 
& &\int_{[\Mbar_{1, 0}(\Pee^1, d[\Pee^1])]} 
     c_{2d} \big ( R^1 (f_{1,0})_*{\rm ev}_1^*\mathcal O_{\Pee^1}(-2) \big )  \\
&=&\int_{[\Mbar_{1, 0}(C, d[C])]} 
     c_{2d} \big ( R^1 (f_{1,0})_*{\rm ev}_1^*(T_S|_C) \big )  \\
&=&\int_{[\Mbar_{1, 0}(S, d[C])]} 
     c_{2d} \big ( R^1 (f_{1,0})_*{\rm ev}_1^*T_S \big ).
\end{eqnarray*}
Note that $R^1 (f_{1,0})_*{\rm ev}_1^*T_S$ is a rank-$2d$ bundle on the $2d$-dimensional 
moduli space $\Mbar_{1, 0}(S, d[C])$ whose virtual dimension is $0$. 
By Proposition~\ref{virtual-prop},
\begin{eqnarray*} 
\int_{[\Mbar_{1, 0}(\Pee^1, d[\Pee^1])]} 
     c_{2d} \big ( R^1 (f_{1,0})_*{\rm ev}_1^*\mathcal O_{\Pee^1}(-2) \big )
= \deg \big ( [\Mbar_{1, 0}(S, d[C])]^\vir \big ).
\end{eqnarray*}
Since $[\Mbar_{g, r}(S, \beta)]^\vir = 0$ whenever $\beta \ne 0$, we obtain
$$
\int_{[\Mbar_{1, 0}(\Pee^1, d[\Pee^1])]} 
c_{2d} \big ( R^1 (f_{1,0})_*{\rm ev}_1^*\mathcal O_{\Pee^1}(-2) \big ) = 0. 
$$

To prove \eqref{0K3.02}, we apply $(f_{1,0})_*{\rm ev}_1^*$ to the exact sequence
$$
0 \to \mathcal O_{\Pee^1}(-2) \to \mathcal O_{\Pee^1}(-1)^{\oplus 2} 
\to \mathcal O_{\Pee^1} \to 0.
$$
Since $(f_{1,0})_*{\rm ev}_1^*\mathcal O_{\Pee^1} 
= \mathcal O_{\Mbar_{1, 0}(\Pee^1, d[\Pee^1])}$ and
$R^1 (f_{1,0})_*{\rm ev}_1^*\mathcal O_{\Pee^1} = \mathcal H^\vee$, we get 
$$
0 \to \mathcal O_{\Mbar_{1, 0}(\Pee^1, d[\Pee^1])}
\to R^1 (f_{1,0})_*{\rm ev}_1^*\mathcal O_{\Pee^1}(-2) 
\to R^1 (f_{1,0})_*{\rm ev}_1^*\mathcal O_{\Pee^1}(-1)^{\oplus 2} \to \mathcal H^\vee \to 0.
$$
Calculating the total Chern class and using \eqref{Mumford}, we see that
\begin{eqnarray}    \label{0K3.1}
   c\big ( R^1 (f_{1,0})_*{\rm ev}_1^*\mathcal O_{\Pee^1}(-2) \big )
&=&c\big ( R^1 (f_{1,0})_*{\rm ev}_1^*\mathcal O_{\Pee^1}(-1)^{\oplus 2} 
          \big )/c\big ( \mathcal H^\vee \big )   \nonumber  \\
&=&c\big ( R^1 (f_{1,0})_*{\rm ev}_1^*\mathcal O_{\Pee^1}(-1)^{\oplus 2} \big ) 
         \cdot (1 + \lambda).
\end{eqnarray} 
Thus, the top Chern class $c_{2d} \big ( R^1 (f_{1,0})_*{\rm ev}_1^*\mathcal O_{\Pee^1}(-2) \big )$ 
is equal to
$$
c_{2d} \big ( R^1 (f_{1,0})_*{\rm ev}_1^*\mathcal O_{\Pee^1}(-1)^{\oplus 2} \big ) 
+ \lambda \cdot c_{2d-1} \big ( R^1 (f_{1,0})_*{\rm ev}_1^*\mathcal O_{\Pee^1}(-1)^{\oplus 2} \big ).
$$
By the Proposition~2 in \cite{GP} and \eqref{0K3.01}, we conclude that
\begin{eqnarray}    \label{0K3.2}
\int_{[\Mbar_{1, 0}(\Pee^1, d[\Pee^1])]} \lambda \cdot
   c_{2d-1} \big ( R^1 (f_{1,0})_*{\rm ev}_1^*\mathcal O_{\Pee^1}(-1)^{\oplus 2} \big ) 
= -{1 \over 12d}.
\end{eqnarray} 
By \eqref{0K3.1} again, $\lambda \cdot c_{2d-1} \big ( 
R^1 (f_{1,0})_*{\rm ev}_1^*\mathcal O_{\Pee^1}(-2) \big ) = \lambda \cdot 
c_{2d-1} \big ( R^1 (f_{1,0})_*{\rm ev}_1^*\mathcal O_{\Pee^1}(-1)^{\oplus 2} \big )$.
Combining this with \eqref{0K3.2}, we obtain our formula \eqref{0K3.02}.
\end{proof}

\begin{lemma}   \label{Hirzebruch}
Let $d \ge 1$, and $V$ be a rank-$2$ bundle over $\Pee^1$. Then,
\begin{eqnarray}    \label{Hirzebruch.0}
\int_{[\Mbar_{1, 0}(\Pee(V), df)]} \lambda \cdot
c_{2d} \big ( R^1 (f_{1,0})_*{\rm ev}_1^*\mathcal O_{\Pee(V)}(-2) \big ) = {\deg(V) \over 12d}.
\end{eqnarray}
\end{lemma}
\noindent
{\it Proof.}
First of all, assume that $\deg(V) = 2k$ for some integer $k$. 
Then $V$ can be deformed to $\mathcal O_{\Pee^1}(k) \oplus \mathcal O_{\Pee^1}(k)
= \big ( \mathcal O_{\Pee^1} \oplus \mathcal O_{\Pee^1} \big ) \otimes 
\mathcal O_{\Pee^1}(k)$. By Lemma~\ref{deformation}, 
\begin{eqnarray*}    
& &\int_{[\Mbar_{1, 0}(\Pee(V), df)]} \lambda \cdot
   c_{2d} \big ( R^1 (f_{1,0})_*{\rm ev}_1^*\mathcal O_{\Pee(V)}(-2) \big )   \\
&=&\int_{[\Mbar_{1, 0}(\Pee^1 \times \Pee^1, df)]} \lambda \cdot
   c_{2d} \Big ( R^1 (f_{1,0})_*{\rm ev}_1^*\big (\mathcal O_{\Pee^1 \times \Pee^1}(-2) 
   \otimes \mathcal O_{\Pee^1 \times \Pee^1}(-2kf) \big )\Big ).
\end{eqnarray*}
Note that $\Mbar_{1, 0}(\Pee^1 \times \Pee^1, df) \cong \Pee^1 \times 
\Mbar_{1, 0}(\Pee^1, d[\Pee^1])$. Thus, we obtain
\begin{eqnarray*}    
& &\int_{[\Mbar_{1, 0}(\Pee(V), df)]} \lambda \cdot
   c_{2d} \big ( R^1 (f_{1,0})_*{\rm ev}_1^*\mathcal O_{\Pee(V)}(-2) \big )   \\
&=&\int_{[\Pee^1 \times \Mbar_{1, 0}(\Pee^1, d[\Pee^1])]} \pi_2^*\lambda \cdot
   c_{2d} \Big ( \pi_2^*\big (R^1 (f_{1,0})_*{\rm ev}_1^*\mathcal O_{\Pee^1}(-2) \big )
   \otimes \pi_1^*\mathcal O_{\Pee^1}(-2k) \Big )
\end{eqnarray*}
where $\pi_1$ and $\pi_2$ are the projection on $\Pee^1 \times \Mbar_{1, 0}(\Pee^1, d[\Pee^1])$. 
Hence,
\begin{eqnarray}    
& &\int_{[\Mbar_{1, 0}(\Pee(V), df)]} \lambda \cdot
   c_{2d} \big ( R^1 (f_{1,0})_*{\rm ev}_1^*\mathcal O_{\Pee(V)}(-2) \big )   \nonumber  \\
&=&\int_{[\Pee^1 \times \Mbar_{1, 0}(\Pee^1, d[\Pee^1])]} \pi_2^*\lambda \cdot
   \pi_2^*c_{2d-1} \big (R^1 (f_{1,0})_*{\rm ev}_1^*\mathcal O_{\Pee^1}(-2) \big )
   \cdot \pi_1^*c_1\big (\mathcal O_{\Pee^1}(-2k) \big )    \nonumber  \\
&=&{\deg(V) \over 12d}         \label{Hirzebruch.1}
\end{eqnarray}
where we have used formula \eqref{0K3.02} in the last step.

Next, assume that $\deg(V) = 2k + 1$ for some integer $k$. Then $V$ can be deformed to 
$\big ( \mathcal O_{\Pee^1}(2) \oplus \mathcal O_{\Pee^1}(-1) \big ) \otimes 
\mathcal O_{\Pee^1}(k)$. As in the previous paragraph, we have
\begin{eqnarray*}    
& &\int_{[\Mbar_{1, 0}(\Pee(V), df)]} \lambda \cdot
   c_{2d} \big ( R^1 (f_{1,0})_*{\rm ev}_1^*\mathcal O_{\Pee(V)}(-2) \big )   \\
&=&\int_{[\Mbar_{1, 0}(S, df)]} \lambda \cdot
   c_{2d} \Big ( R^1 (f_{1,0})_*{\rm ev}_1^*\big (\mathcal O_S(-2) 
   \otimes \mathcal O_S(-2kf) \big )\Big )     \\
&=&\int_{[\Mbar_{1, 0}(S, df)]} \lambda \cdot
   c_{2d} \big ( R^1 (f_{1,0})_*{\rm ev}_1^*\mathcal O_S(-2) \big ) + {2k \over 12d}
\end{eqnarray*}
where $S = \Pee \big ( \mathcal O_{\Pee^1}(2) \oplus \mathcal O_{\Pee^1}(-1) \big )$.
Let $\mathbb F_1$ be the blown-up of $\Pee^2$ at a point $p$, and $\sigma$ be the exceptional curve.
Then, $T_{\mathbb F_1}|_\sigma \cong \mathcal O_{\Pee^1}(2) \oplus \mathcal O_{\Pee^1}(-1)$. So
\begin{eqnarray}    
& &\int_{[\Mbar_{1, 0}(\Pee(V), df)]} \lambda \cdot
   c_{2d} \big ( R^1 (f_{1,0})_*{\rm ev}_1^*\mathcal O_{\Pee(V)}(-2) \big )   \nonumber  \\
&=&\int_{[\Mbar_{1, 0}(\Pee(T_{\mathbb F_1}|_\sigma), df)]} \lambda \cdot
   c_{2d} \big ( R^1 (f_{1,0})_*{\rm ev}_1^*\mathcal O_{\Pee(T_{\mathbb F_1}|_\sigma)}(-2) \big ) 
   + {2k \over 12d}          \nonumber   \\
&=&\int_{[\Mbar_{1, 0}(\Pee(T_{\mathbb F_1}), df)]} \phi^*[\sigma] \cdot \lambda \cdot
   c_{2d} \big ( R^1 (f_{1,0})_*{\rm ev}_1^*\mathcal O_{\Pee(T_{\mathbb F_1})}(-2) \big ) 
   + {2k \over 12d}       \label{Hirzebruch.2}
\end{eqnarray}
where the morphism $\phi: \Mbar_{1, 0}\big (\Pee(T_{\mathbb F_1}), df \big ) \to \mathbb F_1$ 
is from \eqref{phiToB}. Let $f_0$ be a fiber of the ruling $\mathbb F_1 \to \Pee^1$, 
and $C$ be a smooth conic in $\Pee^2$ such that $p \not \in C$.
We use $C$ to denote its strict transform in $\mathbb F_1$. Then, $[\sigma] = [C]/2 - [f_0]$.
By \eqref{Hirzebruch.2},
\begin{eqnarray*}   
& &\int_{[\Mbar_{1, 0}(\Pee(V), df)]} \lambda \cdot
   c_{2d} \big ( R^1 (f_{1,0})_*{\rm ev}_1^*\mathcal O_{\Pee(V)}(-2) \big )  \\
&=&{1 \over 2} \cdot \int_{[\Mbar_{1, 0}(\Pee(T_{\mathbb F_1}), df)]} \phi^*[C] \cdot 
   \lambda \cdot c_{2d} \big ( R^1 (f_{1,0})_*{\rm ev}_1^*\mathcal O_{\Pee(T_{\mathbb F_1})}(-2) 
   \big )      \\
& &- \int_{[\Mbar_{1, 0}(\Pee(T_{\mathbb F_1}), df)]} \phi^*[f_0] \cdot \lambda \cdot
   c_{2d} \big ( R^1 (f_{1,0})_*{\rm ev}_1^*\mathcal O_{\Pee(T_{\mathbb F_1})}(-2) \big )
   + {2k \over 12d}        \\
&=&{1 \over 2} \cdot \int_{[\Mbar_{1, 0}(\Pee(T_{\mathbb F_1}|_C), df)]} 
   \lambda \cdot c_{2d} \big ( R^1 (f_{1,0})_*{\rm ev}_1^*\mathcal O_{\Pee(T_{\mathbb F_1}|_C)}(-2) 
   \big )      \\
& &- \int_{[\Mbar_{1, 0}(\Pee(T_{\mathbb F_1}|_{f_0}), df)]} \lambda \cdot
   c_{2d} \big ( R^1 (f_{1,0})_*{\rm ev}_1^*\mathcal O_{\Pee(T_{\mathbb F_1}|_{f_0})}(-2) \big )
   + {2k \over 12d}.     
\end{eqnarray*}
Note that $\deg \big ( T_{\mathbb F_1}|_C \big ) = 6$ and 
$\deg \big ( T_{\mathbb F_1}|_{f_0} \big ) = 2$. By \eqref{Hirzebruch.1},
\begin{equation}
\int_{[\Mbar_{1, 0}(\Pee(V), df)]} \lambda \cdot
   c_{2d} \big ( R^1 (f_{1,0})_*{\rm ev}_1^*\mathcal O_{\Pee(V)}(-2) \big )  
= {1 \over 2} \cdot {6 \over 12d} - {2 \over 12d} + {2k \over 12d}
= {\deg(V) \over 12d}.  \tag*{$\qed$}
\end{equation}

\begin{proposition}   \label{PropDegV}
Let $d$ be a positive integer. Assume that $V$ is a rank-$2$ vector bundle over 
a smooth projective curve $C$. Then, we have
\begin{eqnarray}    \label{PropDegV.0}
\int_{[\Mbar_{1, 0}(\Pee(V), df)]} \lambda \cdot
c_{2d} \big ( R^1 (f_{1,0})_*{\rm ev}_1^*\mathcal O_{\Pee(V)}(-2) \big ) = {\deg(V) \over 12d}.
\end{eqnarray}
\end{proposition}
\begin{proof}
It is well-known that there exist a rank-$2$ bundle $\mathcal V$ over $\Pee^1 \times C$
and two points $b_1, b_2 \in \Pee^1$ such that $\mathcal V|_{\{b_1\} \times C} = V$
and $\mathcal V|_{\{b_2\} \times C} = (\mathcal O_C \oplus M) \otimes N$
where $M$ and $N$ are line bundles on $C$ with $M$ being very ample. 
As in the proof of Lemma~\ref{Hirzebruch}, we conclude from Lemma~\ref{deformation} that 
\begin{eqnarray}    \label{PropDegV.1}
& &\int_{[\Mbar_{1, 0}(\Pee(V), df)]} \lambda \cdot
   c_{2d} \big ( R^1 (f_{1,0})_*{\rm ev}_1^*\mathcal O_{\Pee(V)}(-2) \big )   \nonumber   \\
&=&\int_{[\Mbar_{1, 0}(\Pee(\mathcal O_C \oplus M), df)]} \lambda \cdot
   c_{2d} \big ( R^1 (f_{1,0})_*{\rm ev}_1^*\mathcal O_{\Pee(\mathcal O_C \oplus M)}(-2) \big )
   + {2 \deg(N) \over 12d}.   \qquad
\end{eqnarray}
Since $M$ is very ample, there exists a morphism $\alpha: C \to \Pee^1$ such that
$M = \alpha^*\mathcal O_{\Pee^1}(1)$. Then, $\mathcal O_C \oplus M 
= \alpha^*\big (\mathcal O_{\Pee^1} \oplus \mathcal O_{\Pee^1}(1) \big )$.
This induces an isomorphism 
$$
\Mbar_{1, 0}(\Pee(\mathcal O_C \oplus M), df) \cong C \times_{\Pee^1} 
\Mbar_{1, 0}\big (\Pee(\mathcal O_{\Pee^1} \oplus \mathcal O_{\Pee^1}(1)), df \big ).
$$
Let $\W \alpha: \Mbar_{1, 0}(\Pee(\mathcal O_C \oplus M), df) \to
\Mbar_{1, 0}\big (\Pee (\mathcal O_{\Pee^1} \oplus \mathcal O_{\Pee^1}(1)), df \big )$
be the projection. Then 
\begin{eqnarray*}     
& &\int_{[\Mbar_{1, 0}(\Pee(\mathcal O_C \oplus M), df)]} \lambda \cdot
   c_{2d} \big ( R^1 (f_{1,0})_*{\rm ev}_1^*\mathcal O_{\Pee(\mathcal O_C \oplus M)}(-2) \big )     \\
&=&\int_{[\Mbar_{1, 0}(\Pee(\mathcal O_C \oplus M), df)]} \W \alpha^*\lambda \cdot
   c_{2d} \Big (\W \alpha^*\big ( R^1 (f_{1,0})_*{\rm ev}_1^*\mathcal O_{\Pee\big (
   \mathcal O_{\Pee^1} \oplus \mathcal O_{\Pee^1}(1) \big )}(-2) \big ) \Big )     \\
&=&\deg(\W \alpha) \cdot \int_{[\Mbar_{1, 0}(\Pee(
   \mathcal O_{\Pee^1} \oplus \mathcal O_{\Pee^1}(1)), df)]} \lambda \cdot
   c_{2d} \big (R^1 (f_{1,0})_*{\rm ev}_1^*\mathcal O_{\Pee(
   \mathcal O_{\Pee^1} \oplus \mathcal O_{\Pee^1}(1))}(-2) \big ). 
\end{eqnarray*}
By Lemma~\ref{Hirzebruch} and noticing $\deg(\W \alpha) = \deg(\alpha) = \deg(M)$, we get
\begin{eqnarray}     \label{PropDegV.2}
\int_{[\Mbar_{1, 0}(\Pee(\mathcal O_C \oplus M), df)]} \lambda \cdot
   c_{2d} \big ( R^1 (f_{1,0})_*{\rm ev}_1^*\mathcal O_{\Pee(\mathcal O_C \oplus M)}(-2) \big )
= {\deg(M) \over 12d}.
\end{eqnarray}
Now our formula \eqref{PropDegV.0} follows immediately from \eqref{PropDegV.1} 
and \eqref{PropDegV.2}.
\end{proof}
\section{\bf The homology classes of curves in Hilbert schemes}
\label{sect_homology}

This section contains some technical lemmas which will be used in 
Subsection~\ref{subsect_n=2Case3}. 
These lemmas deal with the homology classes of curves in Hilbert schemes.

\begin{lemma}   \label{CurveSim}
Let $n \ge 2$ and $X$ be a simply connected surface. Let $\Gamma$ be 
an irreducible curve in the Hilbert scheme $\Xn$. Then, 
\begin{eqnarray}     \label{CurveSim.01}
\Gamma \sim \beta_C + d \beta_n  \quad \in H_2(\Xn, \C)
\end{eqnarray} 
for some effective curve class $C$ (possibly zero) and some integer $d$.
\end{lemma}
\begin{proof}
Let $\pi_1, \pi_2$ be the two projections of $X^{[n]} \times X$, 
and recall the universal codimension-$2$ subscheme ${\mathcal Z}_n$ from \eqref{UnivZn}.
Define 
\beq    \label{CurveSim.1}
{\mathcal Z}_\Gamma = \Gamma \times_{X^{[n]}} {\mathcal Z}_n.
\eeq
Let ${\w \pi}_1 = 
\pi_1|_{{\mathcal Z}_\Gamma}: {\mathcal Z}_\Gamma \to \Gamma$ and ${\w \pi}_2 = 
\pi_2|_{{\mathcal Z}_\Gamma}: {\mathcal Z}_\Gamma \to X$. Define
$$
C_\Gamma = \w \pi_2({\mathcal Z}_\Gamma) \subset X.
$$
Let $C_1, \ldots, C_t$ (possibly $t = 0$) be the irreducible components of $C_\Gamma$ 
such that $\dim C_i = 1$ for all $1 \le i \le t$. Let $m_1, \ldots, m_t$ 
be the degrees of the restrictions of $\pi_2|_{{\mathcal Z}_\Gamma}$ to the reduced 
curves ${\big ( \w \pi_2^{-1}(C_1) \big )}_{\rm red}, \ldots, 
{\big ( \w \pi_2^{-1}(C_t) \big )}_{\rm red}$ respectively.

Fix a very ample curve $H$ in $X$. Let $d_i = C_i \cdot H$ for $1 \le i \le t$.
Choose the curve $H$ such that the following conditions are satisfied:
\begin{enumerate}
\item[$\bullet$] for each $1 \le i \le t$, $H$ intersects $C_i$ transversally at 
  $d_i$ distinct smooth points
  $$
  {\w x}_{i, 1}, \ldots, {\w x}_{i, d_i};
  $$

\item[$\bullet$] for $1 \le i \le t$ and $1 \le j \le d_i$, 
  ${\w \pi}_2^{-1}({\w x}_{i, j})$ consists of distinct smooth points 
  $$
  \big (\xi_{i, j; 1}, {\w x}_{i, j} \big ), \, \ldots, \, 
  \big (\xi_{i, j; m_i}, {\w x}_{i, j} \big ) \, \in \, 
  {\big ( {\w \pi}_2^{-1}(C_i) \big )}_{\rm red}
  $$
  at which the restriction of $\w \pi_1$ to ${\big ( {\w \pi}_2^{-1}(C_i) 
  \big )}_{\rm red}$ is also unramified.
\end{enumerate}
For $1 \le i \le t$, $1 \le j \le d_i$ and $1 \le k \le m_i$, let $\w m_{i,j,k}$ 
be the multiplicity of the unique irreducible component of ${\mathcal Z}_\Gamma$ 
containing the smooth point $\big (\xi_{i, j; k}, {\w x}_{i, j} \big )$. 
Then the contribution of $\big (\xi_{i, j; k}, {\w x}_{i, j} \big )$ to 
$\Gamma \cdot D_H$ is exactly $\w m_{i,j,k}$. Therefore,
\beq    \label{CurveSim.2}
\Gamma \cdot D_H = \sum_{i=1}^t \sum_{j=1}^{d_i} \sum_{k=1}^{m_i} {\w m_{i,j,k}}.
\eeq
Note that $\sum_{k=1}^{m_i} {\w m_{i,j,k}}$ is independent of $j$ and $H$. 
Put $\sum_{k=1}^{m_i} {\w m_{i,j,k}} = e_i$. By \eqref{CurveSim.2},
\begin{eqnarray}    \label{CurveSim.3}
\Gamma \cdot D_H = \sum_{i=1}^t \sum_{j=1}^{d_i} e_i = \sum_{i=1}^t e_i d_i 
= \sum_{i=1}^t e_i (C_i \cdot H) = \left ( \sum_{i=1}^t e_i C_i \right ) \cdot H.
\end{eqnarray}

Since $X$ is simply connected, ${\rm Pic}(\Xn) \cong {\rm Pic}(X) \oplus \Z \cdot (B_n/2)$
by \eqref{PicardH1=0}. By the duality between divisor classes and curve classes, 
$\Gamma \sim \beta_C + d \beta_n$ for some integer $d$ and some class $C \in A_1(\Xn)$.
Combining with \eqref{CurveSim.3}, we get
$$
C \cdot H = (\beta_C + d \beta_n) \cdot D_H = \Gamma \cdot D_H 
= \left ( \sum_{i=1}^t e_i C_i \right ) \cdot H
$$
So $C$ and $\sum_{i=1}^t e_i C_i$ are numerically equivalent divisors on the surface $X$.
Since $X$ is simply connected, we see that $C= \sum_{i=1}^t e_i C_i$ as divisors.
\end{proof}

Next, we study the homology classes of curves in 
$C^{(n)} \subset \Xn$ where $C$ denotes a smooth curve in $X$.
The case when $g_C = 0$ has been settled in \cite{LQZ}. So we will assume $g_C \ge 1$. 
We recall some standard facts about $C^{(n)}$ from \cite{ACGH, BT}. 
For a fixed point $p \in C$, let $\Xi$ denote the divisor 
$p + C^{(n-1)} \subset C^{(n)}$. Let 
$$
{\rm AJ}: C^{(n)} \to {\rm Jac}_n(C)
$$
be the Abel-Jacobi map sending an element $\xi \in C^{(n)}$ to the corresponding
degree-$n$ divisor class in ${\rm Jac}_n(C)$. For an element $\delta \in {\rm Jac}_n(C)$, 
the fiber ${\rm AJ}^{-1}(\delta)$ is the complete line system $|\delta|$.
Let $\mathcal Z_n(C) \subset C^{(n)} \times C$ be the universal divisor, 
and let $\w \pi_1, \w \pi_2$ be the two projections on $C^{(n)} \times C$. 
By the Lemma~2.5 on p.340 of \cite{ACGH} and the Proposition~2.1~(iv) of \cite{BT},
we have
\beq    \label{ACGH1}
c_1\big ( \w \pi_{1*} \mathcal O_{\mathcal Z_n(C)} \big )
= (1 - g_C - n)\Xi + \Theta
\eeq
where $\Theta$ is the pull-back via AJ of a Theta divisor on ${\rm Jac}_n(C)$. 

\begin{lemma}   \label{HomlgCls}
Let $n \ge 2$ and $X$ be a simply connected surface. Let $C$ be a smooth curve in $X$,
and $\Gamma \subset C^{(n)}$ be a curve. Then, 
\begin{eqnarray}     \label{HomlgCls.01}
\Gamma \sim 
(\Xi \cdot \Gamma) \beta_C + \big (-(n+ g_C - 1)(\Xi \cdot \Gamma) + 
(\Theta \cdot \Gamma)\big ) \beta_n  \quad \in H_2(\Xn, \C).
\end{eqnarray} 
In addition, for every line 
$\Gamma_0$ in a positive-dimensional fiber ${\rm AJ}^{-1}(\delta)$, we have
\begin{eqnarray}     \label{HomlgCls.02}
\Gamma_0 \sim \beta_C - (n+ g_C - 1) \beta_n.
\end{eqnarray} 
\end{lemma}
\begin{proof}
(i) Recall the universal codimension-$2$ subscheme $\mathcal Z_n \subset \Xn \times X$, 
and let $\pi_1, \pi_2$ be the two projections on $\Xn \times X$. Then, we have
$$
\big ( \pi_{1*} \mathcal O_{\mathcal Z_n} \big )|_{C^{(n)}} 
= \w \pi_{1*} \mathcal O_{\mathcal Z_n(C)}.
$$
It is well-known that $c_1\big ( \pi_{1*} \mathcal O_{\mathcal Z_n} \big ) = -B_n/2$.
Combining with \eqref{ACGH1}, we obtain
\begin{eqnarray}    \label{HomlgCls.1}
      (-B_n/2) \cdot \Gamma 
&=&c_1\big ( \pi_{1*} \mathcal O_{\mathcal Z_n} \big )|_{C^{(n)}} \cdot \Gamma
     = c_1\big ( \w \pi_{1*} \mathcal O_{\mathcal Z_n(C)} \big ) \cdot \Gamma   
                             \nonumber   \\
&=&(1 - g_C - n)(\Xi \cdot \Gamma) + (\Theta \cdot \Gamma).
\end{eqnarray}

Next, let $\alpha \in H^2(X, \C)$. Then, $D_\alpha|_{C^{(n)}} = (C \cdot \alpha) \Xi$.
Thus, we get
$$
D_\alpha \cdot \Gamma =  (C \cdot \alpha) (\Xi \cdot \Gamma).
$$
By \eqref{BasisH^2} and \eqref{HomlgCls.1}, $\Gamma \sim 
(\Xi \cdot \Gamma) \beta_C + \big ( (1 - g_C - n)(\Xi \cdot \Gamma) + 
(\Theta \cdot \Gamma)\big ) \beta_n$.

(ii) Note that $\Theta \cdot \Gamma_0 = 0$. Also, it is known that 
$\Xi|_{{\rm AJ}^{-1}(\delta)} = \mathcal O_{{\rm AJ}^{-1}(\delta)}(1)$. 
So $\Xi \cdot \Gamma_0 = 1$. By (i), we see immediately that 
$\Gamma_0 \sim \beta_C - (n+ g_C - 1) \beta_n$.
\end{proof}

\section{\bf Gromov-Witten invariants of the Hilbert scheme $\Xtwo$}
\label{sect_n=2}

This section studies the Gromov-Witten invariants of the Hilbert scheme $\Xtwo$.
Using the results from previous sections, 
we will prove Theorems~\ref{Intro-theorem_ii} and \ref{Intro-theorem_iii}.

To begin with, we obtain the following from Corollary~\ref{GenTypeCor}.

\begin{proposition}  \label{n=2Prop}
Let $X$ be a simply connected minimal surface of general type admitting a holomorphic 
differential two-form with irreducible zero divisor. Let $\beta \ne 0$. 
Then all the Gromov-Witten invariants without descendant insertions defined via 
$\Mbar_{g, r}(\Xtwo, \beta)$ vanish, except possibly in the following cases
\begin{enumerate}
\item[{\rm (i)}] $g = 0$ and $\beta = d \beta_2$ for some integer $d > 0$;

\item[{\rm (ii)}] $g = 1$ and $\beta = d \beta_2$ for some integer $d > 0$;

\item[{\rm (iii)}] $K_X^2 = 1, g = 0$ and $\beta = \beta_{K_X} - d \beta_2$ for some integer $d$.
\end{enumerate}
\end{proposition}
\begin{proof}
Cases (i) and (ii) follow from Corollary~\ref{GenTypeCor}~(i) and (ii) respectively.
In the case of Corollary~\ref{GenTypeCor}~(iii), we have $g = 0$ and 
$\beta = d_0 \beta_{K_X} - d \beta_2$ for some rational number $d_0 > 0$ and some integer $d$.
We see from \eqref{ExpDim} that the expected dimension of the moduli space 
$\Mbar_{0, r}(\Xtwo, \beta)$ is equal to 
$$
\mathfrak d = -d_0K_X^2 + 1+ r.
$$
Since $d_0K_X^2$ must be a positive integer, we conclude from the Fundamental Class Axiom 
that all the Gromov-Witten invariants without descendant insertions defined via 
$\Mbar_{0, r}(\Xtwo, \beta)$ vanish except possibly when $d_0 K_X^2 = 1$. Now write 
$K_X = mC_0'$ where $C_0'$ is an irreducible and reduced curve, and $m \ge 1$ is an integer.
By Remark~\ref{RmkThmVanish}, $d_0 = d_0'/m$ for some integer $d_0' \ge 1$. 
Therefore, we obtain
$$
1 = d_0 K_X^2 = d_0' m (C_0')^2.
$$
It follows that $d_0' = m = (C_0')^2 = 1$. Hence $K_X^2 = 1$ and $d_0 = 1$.
\end{proof}

Case~(i) in Proposition~\ref{n=2Prop} can be handled via
the Divisor Axiom of Gromov-Witten theory and the results in \cite{LQ1}. 
By the Divisor Axiom, Case~(ii) in Proposition~\ref{n=2Prop} can be reduced to the invariant
\beq   \label{n=2Case2}
\langle 1 \rangle_{1, d\beta_2}^{\Xtwo}.
\eeq
Similarly, Case~(iii) in Proposition~\ref{n=2Prop} can be reduced to the invariant
\beq   \label{n=2Case3}
\langle 1 \rangle_{0, \,\, \beta_{K_X} - d\beta_2}^{\Xtwo}.
\eeq 
The invariants \eqref{n=2Case2} and \eqref{n=2Case3} will be studied in the next two subsections.
\subsection{\bf Calculation of \eqref{n=2Case2}}
\label{subsect_n=2Case2} 
\par
$\,$

Let $d$ denote a positive integer. In this subsection, we will compute \eqref{n=2Case2} for 
an arbitrary smooth projective surface $X$. For simplicity, we put 
\beq    \label{Mgrd}
\Mbar_{g, r, d} = \Mbar_{g, r}(\Xtwo, d\beta_2).
\eeq

\begin{lemma}   \label{obsbundle1}

\begin{enumerate}
\item[{\rm (i)}] The obstruction sheaf $\mathcal Ob = R^1(f_{1,0})_*({\rm ev}_1^*T_\Xtwo)$ 
over the moduli space ${\Mbar}_{1, 0, d}$ is locally free of rank $2d+2$;

\item[{\rm (ii)}] $[{\Mbar}_{1, 0, d}]^\vir = c_{2d+2}(\mathcal Ob)
 \cap [{\Mbar}_{1, 0, d}]$.
\end{enumerate}
\end{lemma}
\begin{proof}
(i) Recall the evaluation map ${\rm ev}_1: \Mbar_{1, 1, d} \to \Xtwo$ and 
the forgetful map $f_{1, 0}: \Mbar_{1, 1, d} \to {\Mbar}_{1, 0, d}$ from Section~\ref{Stable}. 
Let $u = [\mu: D \to \Xtwo] \in {\Mbar}_{1, 0, d}$. Then
$$
H^1 \big (f^{-1}_{1, 0}(u), ({\rm ev}_1^*T_{\Xtwo})|_{f^{-1}_{1, 0}(u)} \big ) 
\cong H^1(D, \mu^*T_{\Xtwo})
= H^1 \big (D, \mu^*(T_{\Xtwo}|_{\mu(D)}) \big ).
$$
Since $d \ge 1$, $\mu(D) = M_2(x) \cong \Pee^1$ for some point $x \in X$. 
By the results in \cite{LQZ}, 
$$
T_{\Xtwo}|_{\mu(D)} = \mathcal O_{\mu(D)}(2) \oplus 
\mathcal O_{\mu(D)}(-2) \oplus \mathcal O_{\mu(D)} \oplus \mathcal O_{\mu(D)}.
$$
Since the curve $D$ is of genus-$1$, $h^1 \big (D, \mu^*(T_{\Xtwo}|_{\mu(D)}) \big ) = 2d+2$.
It follows that the sheaf $R^1(f_{1,0})_*({\rm ev}_1^*T_\Xtwo)$ is locally free of rank $2d+2$.

(ii) First of all, note that there exists a natural morphism 
$$\phi: {\Mbar}_{1, 0, d} \to X$$ sending an element 
$u = [\mu: D \to \Xtwo] \in {\Mbar}_{1, 0, d}$ to $x \in X$ if $\mu(D) = M_2(x)$. 
The fiber $\phi^{-1}(x)$ over $x \in X$ is $\Mbar_{1,0}(M_2(x), d[M_2(x)] \cong 
\Mbar_{1,0}(\Pee^1, d[\Pee^1])$. So the moduli space ${\Mbar}_{1, 0, d}$ is smooth 
(as a stack) with dimension
$$
\dim {\Mbar}_{1, 0, d} = \dim \Mbar_{1,0}(\Pee^1, d[\Pee^1]) + 2 = 2d + 2.
$$
By \eqref{ExpDim}, the expected dimension of ${\Mbar}_{1, 0, d}$ is $0$. Thus, 
the excess dimension of ${\Mbar}_{1, 0, d}$ is $2d+2$. By (i) and Proposition~\ref{virtual-prop}, 
$[{\Mbar}_{1, 0, d}]^\vir = c_{2d+2}(\mathcal Ob) \cap [{\Mbar}_{1, 0, d}]$.
\end{proof}

Via the inclusion map $B_2 \hookrightarrow \Xtwo$, the evaluation map 
${\rm ev}_1: \Mbar_{1, 1, d} \to \Xtwo$ factors through
a morphism $\widetilde {\rm ev}_1: \Mbar_{1, 1, d} \to B_2$. Also, $B_2 \cong \Pee(T_X^\vee)$.
Let $\rho: B_2 \to X$ be the canonical projection. 
Then, there exists a commutative diagram of morphisms:
\begin{eqnarray}\label{com-diagram2}
\begin{matrix}
\Mbar_{1, 1, d} & \overset{\widetilde {\rm ev}_1}\to  &B_2\\
\quad \downarrow^{f_{1,0}}&&\downarrow^{\rho}\\
\Mbar_{1, 0, d} &\overset{\phi} \to&X.       
\end{matrix}  
\end{eqnarray}

\begin{lemma}   \label{obsbundle2}
\begin{enumerate}
\item[{\rm (i)}] Let $\mathcal H$ be the Hodge bundle over $\Mbar_{1, 0, d}$. Then, 
$$
R^1 (f_{1,0})_*\widetilde {\rm ev}_1^*T_{B_2} \cong \mathcal H^\vee \otimes \phi^*T_X;
$$

\item[{\rm (ii)}] There exists an exact sequence of locally free sheaves:
\begin{eqnarray}   \label{obsbundle2.0}
0 \to R^1 (f_{1,0})_*\widetilde {\rm ev}_1^*T_{B_2} \to \mathcal Ob
\to R^1 (f_{1,0})_*\widetilde {\rm ev}_1^*\mathcal O_{B_2}(-2) \to 0.
\end{eqnarray}
\end{enumerate}
\end{lemma}
\begin{proof}
(i) Let $T_{B_2/X}$ be the relative tangent sheaf for the projection $\rho: B_2 \to X$. 
Applying the functors $\widetilde {\rm ev}_1^*$ and $(f_{1,0})_*$ to the exact sequence
$$
0 \to T_{B_2/X} \to T_{B_2} \to \rho^*T_X \to 0
$$
of locally free sheaves, we obtain an exact sequence
\begin{eqnarray}   \label{obsbundle2.1}
R^1 (f_{1,0})_*\widetilde {\rm ev}_1^*T_{B_2/X} \to R^1 (f_{1,0})_*\widetilde {\rm ev}_1^*T_{B_2} 
\to R^1 (f_{1,0})_*\widetilde {\rm ev}_1^*(\rho^*T_X) \to 0.
\end{eqnarray}
where we have used $R^2(f_{1,0})_*\widetilde {\rm ev}_1^*T_{B_2/X} = 0$
since $f_{1,0}$ is of relative dimension $1$.

We claim that $R^1 (f_{1,0})_*\widetilde {\rm ev}_1^*T_{B_2/X} = 0$.
Indeed, let $u = [\mu: D \to \Xtwo] \in {\Mbar}_{1, 0, d}$, and assume that $\mu(D) = M_2(x)$.
Since $T_{B_2/X}|_{M_2(x)} = T_{M_2(x)} = \mathcal O_{M_2(x)}(2)$, 
$$
H^1 \big (f^{-1}_{1, 0}(u), \widetilde {\rm ev}_1^*T_{B_2/X}|_{f^{-1}_{1, 0}(u)} \big ) 
\cong H^1(D, \mu^*\mathcal O_{M_2(x)}(2)) = 0.
$$

By \eqref{obsbundle2.1}, $R^1 (f_{1,0})_*\widetilde {\rm ev}_1^*T_{B_2} 
\cong R^1 (f_{1,0})_*\widetilde {\rm ev}_1^*(\rho^*T_X)$.
Since $\rho \circ \widetilde {\rm ev}_1 = \phi \circ f_{1,0}$, we get
\begin{eqnarray*}
       R^1 (f_{1,0})_*\widetilde {\rm ev}_1^*T_{B_2}
&\cong&R^1 (f_{1,0})_*\big ( f_{1,0}^*(\phi^*T_X) \big )   \\
&\cong&R^1 (f_{1,0})_*\mathcal O_{\Mbar_{1, 1, d}} \otimes \phi^*T_X  \\
&\cong&\mathcal H^\vee \otimes \phi^*T_X.
\end{eqnarray*}

(ii) Since ${\rm ev}_1$ factors through $\widetilde {\rm ev}_1$, 
we see from Lemma~\ref{obsbundle1}~(i) that
$$
\mathcal Ob = R^1(f_{1,0})_*({\rm ev}_1^*T_\Xtwo)
= R^1(f_{1,0})_*\big (\widetilde {\rm ev}_1^*(T_\Xtwo|_{B_2}) \big ).
$$
Since $B_2$ is a smooth divisor in $\Xtwo$ and $\mathcal O_{B_2}(B_2) = \mathcal O_{B_2}(-2)$, we have
$$
0 \to T_{B_2} \to T_\Xtwo|_{B_2} \to \mathcal O_{B_2}(-2) \to 0.
$$
Applying the functors $\widetilde {\rm ev}_1^*$ and $(f_{1,0})_*$, we obtain an exact sequence 
\begin{eqnarray}   \label{obsbundle2.2}
(f_{1,0})_*\widetilde {\rm ev}_1^*\mathcal O_{B_2}(-2) \to 
R^1 (f_{1,0})_*\widetilde {\rm ev}_1^*T_{B_2} \to \mathcal Ob
\to R^1 (f_{1,0})_*\widetilde {\rm ev}_1^*\mathcal O_{B_2}(-2) \to 0.
\end{eqnarray}

We claim that $(f_{1,0})_*\widetilde {\rm ev}_1^*\mathcal O_{B_2}(-2) = 0$.
Indeed, let $u = [\mu: D \to \Xtwo] \in {\Mbar}_{1, 0, d}$, and assume that $\mu(D) = M_2(x)$.
Then, since $\mathcal O_{B_2}(-2)|_{M_2(x)} = \mathcal O_{M_2(x)}(-2)$, 
$$
H^0 \big (f^{-1}_{1, 0}(u), \widetilde {\rm ev}_1^*\mathcal O_{B_2}(-2)|_{f^{-1}_{1, 0}(u)} \big ) 
\cong H^0(D, \mu^*\mathcal O_{M_2(x)}(-2)) = 0.
$$
So $(f_{1,0})_*\widetilde {\rm ev}_1^*\mathcal O_{B_2}(B_2) = 0$, and 
\eqref{obsbundle2.2} is simplified to the exact sequence \eqref{obsbundle2.0}.
\end{proof}

\begin{theorem}   \label{theorem_ii}
Let $d \ge 1$. Let $X$ be a smooth projective surface. Then, 
$$
\langle 1 \rangle_{1, d\beta_2}^{\Xtwo} = {K_X^2 \over 12d}.
$$
\end{theorem}
\begin{proof}
By Lemma~\ref{obsbundle1}~(ii), $\langle 1 \rangle_{1, d\beta_2}^{\Xtwo} = 
\deg [{\Mbar}_{1, 0, d}]^\vir = \deg \big (c_{2d+2}(\mathcal Ob)  \cap [{\Mbar}_{1, 0, d}] \big )$. 
The Hodge bundle $\mathcal H$ and $R^1 (f_{1,0})_*\widetilde {\rm ev}_1^*\mathcal O_{B_2}(-2)$
on the moduli space $\Mbar_{1, 0, d}$ are of ranks $1$ and $2d$ respectively.
Therefore, by Lemma~\ref{obsbundle2} and \eqref{Mumford}, 
\begin{eqnarray*}            
   \langle 1 \rangle_{1, d\beta_2}^{\Xtwo} 
&=&c_2 \big ( \mathcal H^\vee \otimes \phi^*T_X \big ) \, \cdot \, c_{2d} \big ( 
       R^1 (f_{1,0})_*\widetilde {\rm ev}_1^*\mathcal O_{B_2}(-2) \big )   \nonumber   \\
&=&\big ( \lambda^2 + \phi^*K_X \cdot \lambda + \phi^*c_2(T_X) \big ) \, \cdot \, 
       c_{2d} \big ( R^1 (f_{1,0})_*\widetilde {\rm ev}_1^*\mathcal O_{B_2}(-2) \big )  \\
&=&\big ( \phi^*K_X \cdot \lambda + \phi^*c_2(T_X) \big ) \, \cdot \, 
       c_{2d} \big ( R^1 (f_{1,0})_*\widetilde {\rm ev}_1^*\mathcal O_{B_2}(-2) \big ).
\end{eqnarray*}
Let $\chi(X)$ be the Euler characteristic of $X$. By \eqref{0K3.01}, we obtain
\begin{eqnarray*} 
& &\phi^*c_2(T_X) \cdot c_{2d} \big ( 
   R^1 (f_{1,0})_*\widetilde {\rm ev}_1^*\mathcal O_{B_2}(-2) \big )   \\
&=&\chi(X) \cdot \int_{[\Mbar_{1, 0}(\Pee^1, d[\Pee^1])]} 
   c_{2d} \big ( R^1 (f_{1,0})_*{\rm ev}_1^*\mathcal O_{\Pee^1}(-2) \big )   \\
&=&0.
\end{eqnarray*} 
Hence, $\langle 1 \rangle_{1, d\beta_2}^{\Xtwo} = \phi^*K_X \cdot \lambda \cdot  
c_{2d} \big ( R^1 (f_{1,0})_*\widetilde {\rm ev}_1^*\mathcal O_{B_2}(-2) \big )$.
Choose two smooth irreducible curves $C_1$ and $C_2$ satisfying $[C_1] - [C_2] = K_X$. 
Since $B_2 \cong \Pee(T_X^\vee)$,
\begin{eqnarray*}            
   \langle 1 \rangle_{1, d\beta_2}^{\Xtwo}    
&=&\phi^*([C_1] - [C_2]) \cdot \lambda \cdot  
   c_{2d} \big ( R^1 (f_{1,0})_*\widetilde {\rm ev}_1^*\mathcal O_{B_2}(-2) \big )   \\
&=&\int_{[\Mbar_{1, 0}(\Pee(T_X^\vee|_{C_1}), df)]} \lambda \cdot
   c_{2d} \big ( R^1 (f_{1,0})_*{\rm ev}_1^*\mathcal O_{\Pee(T_X^\vee|_{C_1})}(-2) \big ) \\
& &- \int_{[\Mbar_{1, 0}(\Pee(T_X^\vee|_{C_2}), df)]} \lambda \cdot
   c_{2d} \big ( R^1 (f_{1,0})_*{\rm ev}_1^*\mathcal O_{\Pee(T_X^\vee|_{C_2})}(-2) \big ) \\
&=&{\deg(T_X^\vee|_{C_1}) \over 12d} - {\deg(T_X^\vee|_{C_2}) \over 12d}   \\
&=&{K_X^2 \over 12d}
\end{eqnarray*}
where we have used Proposition~\ref{PropDegV} in the third step.
\end{proof}

\subsection{\bf Calculation of \eqref{n=2Case3}}
\label{subsect_n=2Case3}
\par
$\,$

Let $X$ be a minimal surface of general type with 
$K_X^2 = 1$ and $p_g \ge 1$. By Noether's inequality, $p_g \le 2$. 
Thus, $p_g = 1$ or $2$. If $p_g = 2$, then we see from the proof 
of Proposition~(8.1) in \cite{BPV} that
$|K_X|$ is a pencil without fixed part and with one base point, 
and the general canonical curve is a genus-$2$ smooth irreducible curve.
If $p_g = 1$, then $|K_X|$ consists of a single element which is 
a connected curve of arithmetic genus-$2$.
In this subsection, we will study \eqref{n=2Case3} by assuming
that $X$ is a simply connected minimal surface of general type with 
$K_X^2 = 1$ and $1 \le p_g \le 2$, and that every member in $|K_X|$ is 
a smooth irreducible curve (of genus-$2$). 

Our first lemma asserts that the lower bound of $d$ for the class 
$\beta_{K_X} + d \beta_2$ to be effective is equal to $-3$,
and classifies all the curves homologous to $\beta_{K_X} - 3 \beta_2$.

\begin{lemma}  \label{LowerBd}
Let $X$ be a simply connected minimal surface of general type with 
$K_X^2 = 1$ and $1 \le p_g \le 2$ such that every member in $|K_X|$ is 
a smooth curve. Let $\Gamma$ be a curve in $\Xtwo$ such that 
$\Gamma \sim \beta_{K_X} + d \beta_2$ for some integer $d$. Then, 
\begin{enumerate}
\item[{\rm (i)}] $d \ge -3$. 

\item[{\rm (ii)}]  $d = -3$ if and only if 
$\Gamma = g_2^1(C) \subset C^{(2)}$, where $C \in |K_X|$ and
$g_2^1(C)$ is the unique linear system of 
dimension $1$ and degree $2$ on the genus-$2$ curve $C$.
\end{enumerate}
\end{lemma}
\begin{proof}
Let $\Gamma_1, \ldots, \Gamma_t$ be the irreducible components of $\Gamma$.
By Lemma~\ref{CurveSim}, there exists $t_1$ with $1 \le t_1 \le t$ such that
for $1 \le i \le t_1$, $\Gamma_i \sim \beta_{C_i} + d_i \beta_2$ for some curve $C_i$
and integer $d_i$, and that for $t_1 < j \le t$, $\Gamma_j \sim d_j \beta_2$ for 
some integer $d_j > 0$. Then, 
\begin{eqnarray}     \label{LowerBd.1}
\beta_{K_X} + d \beta_2 \sim \Gamma = \sum_{i=1}^t \Gamma_i 
\sim \sum_{i=1}^{t_1} (\beta_{C_i} + d_i \beta_2) + \sum_{j=t_1+1}^t (d_j \beta_2).
\end{eqnarray}
So $K_X \sim \sum_{i=1}^{t_1} C_i$. Since $X$ is simply connected,
$K_X = \sum_{i=1}^{t_1} C_i$ as divisors. 
Since every member in $|K_X|$ is a smooth irreducible curve, 
$t_1 = 1$ and $C_1 \in |K_X|$. By \eqref{LowerBd.1}, $d = d_1 + \sum_{j=2}^t d_j \ge d_1$.
To prove the lemma, it suffices to prove $d_1 \ge -3$, i.e., we will assume 
in the rest of the proof that $\Gamma = \Gamma_1$ is irreducible.

We claim that there exists a non-empty open subset $U$ of $\Gamma$ 
such that every $\xi \in U$ consists of two distinct points in $X$.
Indeed, assume that this is not true. Then $\Gamma \subset B_2$. 
Let $\alpha \in H^2(X, \C)$. By abusing notations, we also use $\alpha$ 
to denote a real surface in $X$ representing the cohomology class $\alpha$. 
Note that $D_\alpha|_{B_2} = D_\alpha|_{M_2(X)} = 2M_2(\alpha)$. Thus, we obtain
$$
K_X \cdot \alpha = (\beta_{K_X} + d \beta_2) \cdot D_\alpha
= \Gamma \cdot D_\alpha = \Gamma \cdot D_\alpha|_{B_2} = 2 \Gamma \cdot M_2(\alpha).
$$
So $K_X \cdot \alpha$ is even for every $\alpha \in H^2(X, \C)$. 
This contradicts to $K_X^2 = 1$.

Next, we claim that either $\Gamma \sim \beta_{K_X}$, 
or there exists an irreducible curve $C \subset X$ such that 
$\Supp(\xi) \in C$ for every $\xi \in \Gamma$.
Assume that there is no irreducible curve $C \subset X$ such that 
$\Supp(\xi) \in C$ for every $\xi \in \Gamma$. Then by the previous paragraph,
we conclude that there exist a non-empty open subset $U$ of $\Gamma$ and 
two distinct irreducible curves $\W C_1, \W C_2 \subset X$ such that every $\xi \in U$ 
is of the form $x_1 + x_2$ with $x_1 \in \W C_1$, $x_2 \in \W C_2$ and $x_1 \ne x_2$.
This leads to two (possibly constant) rational maps 
$f_1: \Gamma \to \W C_1$ and $f_2: \Gamma \to \W C_2$.
Choose the real surface $\alpha \subset X$ in the previous paragraph such that
$\alpha, \W C_1$ and $\W C_2$ are in general position. Then, 
$$
K_X \cdot \alpha = \Gamma \cdot D_\alpha 
= \deg(f_1) \W C_1 \cdot \alpha + \deg(f_2) \W C_2 \cdot \alpha.
$$
So $K_X = \deg(f_1) \W C_1 + \deg(f_2) \W C_2$ in $H^2(X, \C)$. Since $X$ is simply-connected,
$K_X = \deg(f_1) \W C_1 + \deg(f_2) \W C_2$ as divisors. 
Since every member in $|K_X|$ is a smooth irreducible curve and $K_X^2 = 1$,
we must have an equality of sets:
$$
\{ \deg(f_1), \deg(f_2) \} = \{1, 0\}.
$$
For simplicity, 
let $\deg(f_1) = 1$ and $\deg(f_2) = 0$. Then, $x_2 \in \Supp(\xi)$ for every $\xi \in \Gamma$,
and $\W C_1 \in |K_X|$. By our assumption, $x_2 \not \in \W C_1$ (otherwise, 
$\Supp(\xi) \in \W C_1$ for every $\xi \in \Gamma$). Thus, $\Gamma = \W C_1 + x_2$. 
So $\Gamma \sim \beta_{\W C_1} = \beta_{K_X}$.

Finally, we assume that there exists an irreducible curve $C \subset X$ such that 
$\Supp(\xi) \in C$ for every $\xi \in \Gamma$.
We claim that $C \in |K_X|$. Note that there exists a positive integer $s$
such that for a general point $x \in C$, there exist $s$ distinct elements 
$\xi_1, \ldots, \xi_s \in \Gamma$ such that $x \in \Supp(\xi_i)$ for every $i$.
Choose the real surface $\alpha \subset X$ such that $\alpha$ and $C$ are 
in general position. Then, we have
$$
K_X \cdot \alpha = s C \cdot \alpha.
$$
So $K_X = sC$ in $H^2(X, \C)$. Since $X$ is simply connected, $K_X = sC$ as divisors. 
Since $K_X^2 = 1$, $s = 1$ and $C \in |K_X|$. Since $\Gamma \in C^{(2)}$ and $g_C = 2$,
$$
\beta_{K_X} + d \beta_2 \sim \Gamma \sim 
(\Xi \cdot \Gamma) \beta_{K_X} + \big (-3(\Xi \cdot \Gamma) + 
(\Theta \cdot \Gamma)\big ) \beta_2
$$
by \eqref{HomlgCls.01}.
Thus, $\Xi \cdot \Gamma = 1$ and $d = -3 + (\Theta \cdot \Gamma)$.
Since $\Theta$ is a nef divisor of $C^{(2)}$, we have $d \ge -3$.
In addition, $d = -3$ if and only if $\Theta \cdot \Gamma = 0$.
Since $\Theta$ is the pull-back of a Theta divisor on ${\rm Jac}_2(C)$ via 
the map ${\rm AJ}: C^{(2)} \to {\rm Jac}_2(C)$, $d = -3$ if and only if 
$\Gamma$ is contracted to a point by ${\rm AJ}$. 
Note that the only positive-dimensional fiber of the map ${\rm AJ}$ is $g_2^1(C) \cong \Pee^1$.
Hence $d = -3$ if and only if $\Gamma = g_2^1(C) \subset C^{(2)}$.
\end{proof}

Fix $C \in |K_X|$ and an isomorphism $\mu_0: \Pee^1 \to g_2^1(C)$. 
Via the inclusion $g_2^1(C) \subset C^{(2)}\subset \Xtwo$, 
regard $\mu_0$ as a morphism from $\Pee^1$ to $\Xtwo$. Then, we have
\beq   \label{DefOfMu0}
[\mu_0: \Pee^1 \to \Xtwo] \in \Mbar_{0, 0}(\Xtwo, \beta_{K_X} - 3\beta_2).
\eeq

\begin{lemma}   \label{mu0}
$h^1(\Pee^1, \mu_0^*T_{\Xtwo}) = p_g - 1$.
\end{lemma}
\par\noindent
{\it Proof.}
From the exact sequence $0 \to \mathcal O_X \to \mathcal O_X(K_X) \to 
\mathcal O_C(K_X) \to 0$, we get
$$
0 \to H^0(X, \mathcal O_X) \to H^0(X, \mathcal O_X(K_X)) \to 
H^0(C, \mathcal O_C(K_X)) \to H^1(X, \mathcal O_X).
$$
Since $X$ is a simply connected surface, $h^1(X, \mathcal O_X) = 0$. So 
\beq \label{mu0.1}
h^0(C, \mathcal O_C(K_X)) = p_g - 1.
\eeq

Put $\Gamma = g_2^1(C) \subset C^{(2)} \subset \Xtwo$. Since the smooth rational curve $\Gamma$ 
is the only positive-dimensional fiber of ${\rm AJ}: C^{(2)} \to {\rm Jac}_2(C)$, 
$\Gamma$ is a $(-1)$-curve contracted by ${\rm AJ}$. 
So the normal bundle $N_{\Gamma \subset C^{(2)}}$ of $\Gamma$ in $C^{(2)}$ is given by:
\beq \label{mu0.2}
N_{\Gamma \subset C^{(2)}} = \mathcal O_\Gamma(-1).
\eeq
Since $T_\Gamma = \mathcal O_\Gamma(2)$, we see from 
$0 \to T_\Gamma \to T_{C^{(2)}}|_\Gamma \to N_{\Gamma \subset C^{(2)}} \to 0$ that 
\beq \label{mu0.2.1}
T_{C^{(2)}}|_\Gamma = \mathcal O_\Gamma(2) \oplus \mathcal O_\Gamma(-1).
\eeq
Since $K_{\Xtwo} \cdot \Gamma 
= D_{K_X} \cdot (\beta_{K_X} - 3\beta_2) = K_X^2 = 1$, we have
$$
\deg N_{\Gamma \subset \Xtwo} = -K_{\Xtwo} \cdot \Gamma - \deg T_\Gamma = -3.
$$
So $\deg \big (N_{C^{(2)} \subset \Xtwo}|_\Gamma \big ) = -2$ in view of 
\eqref{mu0.2} and the exact sequence
\beq \label{mu0.3}
0 \to N_{\Gamma \subset C^{(2)}} \to N_{\Gamma \subset \Xtwo} \to 
N_{C^{(2)} \subset \Xtwo}|_\Gamma \to 0.
\eeq

We claim that $N_{C^{(2)} \subset \Xtwo}|_\Gamma = \mathcal O_\Gamma(-p_g) \oplus 
\mathcal O_\Gamma(p_g-2)$. Since $N_{C^{(2)} \subset \Xtwo}|_\Gamma$ is 
a degree-$(-2)$ rank-$2$ bundle on $\Gamma \cong \Pee^1$, it suffices to prove 
\beq \label{mu0.4}
h^0 \big (\Gamma, N_{C^{(2)} \subset \Xtwo}|_\Gamma \big ) = p_g-1.
\eeq
It is known from \cite{AIK} that $N_{C^{(2)} \subset \Xtwo} = \pi_{1*}\pi_2^*\mathcal O_X(C)|_{C^{(2)}}
= \pi_{1*}\pi_2^*\mathcal O_X(K_X)|_{C^{(2)}}$
where $\pi_1: {\mathcal Z}_2 \to \Xtwo$ and $\pi_2: {\mathcal Z}_2 \to X$ are 
the natural projections. Let ${\mathcal Z}_\Gamma = \pi_1^{-1}(\Gamma) \subset {\mathcal Z}_2$.
Note that $\pi_2 \big ( {\mathcal Z}_\Gamma \big ) = C$.
Put $\w \pi_1 = \pi_1|_{{\mathcal Z}_\Gamma}: {\mathcal Z}_\Gamma \to \Gamma$
and $\w \pi_2 = \pi_2|_{{\mathcal Z}_\Gamma}: {\mathcal Z}_\Gamma \to C$.
Then, $\w \pi_2$ is an isomorphism. Up to an isomorphism, $\w \pi_1$ is the double cover
$C \to \Pee^1$ corresponding to the linear system $g_2^1(C)$. We have
\beq \label{mu0.401}
N_{C^{(2)} \subset \Xtwo}|_\Gamma = \pi_{1*}\pi_2^*\mathcal O_X(K_X)|_\Gamma 
= \w \pi_{1*}\w \pi_2^*\big ( \mathcal O_X(K_X)|_C \big )
= \w \pi_{1*}\w \pi_2^*\mathcal O_C(K_X).
\eeq
Thus $H^0 \big (\Gamma, N_{C^{(2)} \subset \Xtwo}|_\Gamma \big )
= H^0 \big (\Gamma, \w \pi_{1*}\w \pi_2^*\mathcal O_C(K_X) \big )
= H^0 \big ({\mathcal Z}_\Gamma, \w \pi_2^*\mathcal O_C(K_X) \big )$.
Since $\w \pi_2$ is an isomorphism, we conclude from \eqref{mu0.1} that 
$$
H^0 \big (\Gamma, N_{C^{(2)} \subset \Xtwo}|_\Gamma \big )
\cong H^0(C, \mathcal O_C(K_X)) = p_g-1.
$$
This proves \eqref{mu0.4}. Therefore, we obtain 
\beq \label{mu0.5}
N_{C^{(2)} \subset \Xtwo}|_\Gamma = \mathcal O_\Gamma(-p_g) \oplus \mathcal O_\Gamma(p_g - 2).
\eeq

Now the exact sequence \eqref{mu0.3} becomes 
$$
0 \to \mathcal O_\Gamma(-1) \to N_{\Gamma \subset \Xtwo} \to 
\mathcal O_\Gamma(-p_g) \oplus \mathcal O_\Gamma(p_g - 2) \to 0
$$
which splits since $1 \le p_g \le 2$.
So $N_{\Gamma \subset \Xtwo} = \mathcal O_\Gamma(-1) \oplus \mathcal O_\Gamma(-p_g) 
\oplus \mathcal O_\Gamma(p_g - 2)$. 
From $T_\Gamma = \mathcal O_\Gamma(2)$ and the exact sequence
$
0 \to T_\Gamma \to T_{\Xtwo}|_\Gamma \to N_{\Gamma \subset \Xtwo} \to 0,
$
we see that $T_{\Xtwo}|_\Gamma = \mathcal O_\Gamma(2) \oplus \mathcal O_\Gamma(-1) 
\oplus \mathcal O_\Gamma(-p_g) \oplus \mathcal O_\Gamma(p_g - 2)$. Finally, 
\begin{equation}
h^1(\Pee^1, \mu_0^*T_{\Xtwo}) = h^1 \big (\Pee^1, \mu_0^*(T_{\Xtwo}|_\Gamma) \big )
= p_g - 1.
\tag*{$\qed$}
\end{equation}

\begin{theorem}   \label{theorem_iii}
Let $X$ be a simply connected minimal surface of general type with 
$K_X^2 = 1$ and $1 \le p_g \le 2$ such that every member in $|K_X|$ is smooth. Then,
\begin{enumerate}
\item[{\rm (i)}] $\Mbar_{0, 0}(\Xtwo, \beta_{K_X} - 3\beta_2) \cong |K_X| \cong \Pee^{p_g-1}$;

\item[{\rm (ii)}] $\langle 1 \rangle_{0, \,\, \beta_{K_X} - 3 \beta_2}^{\Xtwo} 
= (-1)^{\chi(\mathcal O_X)}$.
\end{enumerate}
\end{theorem}
\par\noindent
{\it Proof.}
For simplicity, we denote $\Mbar_{0, 0}(\Xtwo, \beta_{K_X} - 3\beta_2)$ by $\Mbar$.

(i) Let $[\mu: D \to \Xtwo] \in \Mbar$. Put $\Gamma = \mu(D)$.
As in the first paragraph in the proof of Lemma~\ref{LowerBd}, let $\Gamma_1, \ldots, \Gamma_t$ be 
the irreducible components of $\Gamma$. Let $m_i$ be the degree of the restriction
$\mu|_{\mu^{-1}(\Gamma_i)}: \mu^{-1}(\Gamma_i) \to \Gamma_i$. Then, 
\beq   \label{prop_iii.1}
\beta_{K_X} - 3 \beta_2 = \mu_*[D] = \sum_{i=1}^t m_i [\Gamma_i].
\eeq
By Lemma~\ref{CurveSim}, there exists some $t_1$ with $1 \le t_1 \le t$ such that
for $1 \le i \le t_1$, $\Gamma_i \sim \beta_{C_i} + d_i \beta_2$ for some curve $C_i$
and integer $d_i$, and that for $t_1 < j \le t$, $\Gamma_j \sim d_j \beta_2$ for 
some integer $d_j > 0$. Combining with \eqref{prop_iii.1}, we obtain
$$
\beta_{K_X} - 3 \beta_2 
= \sum_{i=1}^{t_1} m_i(\beta_{C_i} + d_i \beta_2) + \sum_{j=t_1+1}^t m_j(d_j \beta_2).
$$
So $K_X \sim \sum_{i=1}^{t_1} m_i C_i$. Since $X$ is simply connected,
$K_X = \sum_{i=1}^{t_1} m_i C_i$ as divisors. 
Since every member in $|K_X|$ is smooth, 
$t_1 = 1$, $m_1 = 1$ and $C_1 \in |K_X|$. By Lemma~\ref{LowerBd}, $t=1$ and 
$\Gamma_1 = g_2^1(C_1)$. So $\mu(D) = g_2^1(C_1)$. Since $m_1 = 1$ and there is no marked points, 
$D = \Pee^1$ and $\mu$ is an isomorphism from $D = \Pee^1$ to $g_2^1(C_1)$. 
Therefore, $\Mbar = |K_X|$ as sets. Since the stable map $\mu: D \to \Xtwo$ has 
the trivial automorphism group, $\Mbar$ is a fine moduli space.

Next, we construct a morphism $\psi: |K_X| \to \Mbar$. Let $\mathcal C \subset |K_X| \times X$
be the family of curves parametrized by $|K_X|$. Then we have the relative Hilbert schemes
$(\mathcal C/|K_X|)^{[2]} \subset (|K_X| \times X/|K_X|)^{[2]}$, i.e.,
$(\mathcal C/|K_X|)^{(2)} \subset |K_X| \times \Xtwo$. Let 
$$
\W \Psi: (\mathcal C/|K_X|)^{(2)} \to \Xtwo
$$ 
be the composition of the inclusion $(\mathcal C/|K_X|)^{(2)} \subset |K_X| \times \Xtwo$ 
and the projection $|K_X| \times \Xtwo \to \Xtwo$. In the relative Jacobian 
${\rm Jac}_2(\mathcal C/|K_X|)$, let $\Sigma$ be the section to the natural projection
${\rm Jac}_2(\mathcal C/|K_X|) \to |K_X|$ such that for $C \in |K_X|$, the point 
$\Sigma(C) = K_C \in {\rm Jac}_2(C)$. Then the natural map $(\mathcal C/|K_X|)^{(2)}
\to {\rm Jac}_2(\mathcal C/|K_X|)$ is the blowing-up of ${\rm Jac}_2(\mathcal C/|K_X|)$
along $\Sigma$. Let $\mathcal E \subset (\mathcal C/|K_X|)^{(2)}$ be 
the exceptional divisor of this blowing-up. Put 
$$
\Psi = \W \Psi|_{\mathcal E}: \mathcal E \to \Xtwo.
$$
Then $\Psi$ is a family of stable maps in $\Mbar$ parametrized by $|K_X|$. 
By the universality of the moduli space $\Mbar$, $\Psi$ induces a morphism $\psi: |K_X| \to \Mbar$.
By the discussion in the previous paragraph, we conclude that $\psi$ is bijective.

By \eqref{ExpDim}, the expected dimension of $\Mbar$ is $0$.
Since $\dim \Mbar = \dim |K_X| = p_g-1$, the excess dimension of $\Mbar$ is equal to $p_g-1$.
By Lemma~\ref{mu0}, $R^1(f_{1,0})_*{\rm ev}_1^*T_{\Xtwo}$ is a rank-$(p_g-1)$ locally free sheaf,
where $f_{1,0}: \Mbar_{0, 1}(\Xtwo, \beta_{K_X} - 3\beta_2) \to \Mbar$ is the forgetful map and
${\rm ev}_1: \Mbar_{0, 1}(\Xtwo, \beta_{K_X} - 3\beta_2) \to \Xtwo$ is the evaluation map.
By Proposition~\ref{virtual-prop}, 
$\Mbar$ is smooth (as a scheme since it is a fine moduli space). 
By Zariski's Main Theorem, the bijective morphism $\psi: |K_X| \to \Mbar$ is an isomorphism.
So $\Mbar \cong |K_X| \cong \Pee^{p_g-1}$. Note also from Proposition~\ref{virtual-prop} that 
\beq            \label{prop_iii.2}
[\Mbar]^{\text{\rm vir}} = c_{p_g-1} \big (R^1(f_{1,0})_*{\rm ev}_1^*T_{\Xtwo} \big ) \cap [\Mbar].
\eeq

(ii) Since $X$ is simply connected, we have $\chi(\mathcal O_X) = 1 + p_g$.
When $p_g = 1$, $\Mbar_{0, 0}(\Xtwo, \beta_{K_X} - 3\beta_2)$ is a smooth point;
so $\langle 1 \rangle_{0, \,\, \beta_{K_X} - 3 \beta_2}^{\Xtwo} = 1$, and our formula holds.
In the rest of the proof, let $p_g = 2$. We will prove that 
$\langle 1 \rangle_{0, \,\, \beta_{K_X} - 3 \beta_2}^{\Xtwo} = -1$.

We adopt the notations from the proof of (i). For simplicity, put $\Mbar = |K_X|$.
Then we have identifications $\Mbar_{0, 1}(\Xtwo, \beta_{K_X} - 3\beta_2) = \mathcal E$
and ${\rm ev}_1 = \Psi$. Moreover, the forgetful map 
$f_{1,0}: \Mbar_{0, 1}(\Xtwo, \beta_{K_X} - 3\beta_2) \to \Mbar$
is identified with the natural projection $f: \mathcal E \to |K_X|$.
In view of \eqref{prop_iii.2}, we have 
\beq            \label{prop_iii.3}
[\Mbar]^{\text{\rm vir}} = c_1 \big (R^1f_*\Psi^*T_{\Xtwo} \big ) \cap [\Mbar].
\eeq

To understand the line bundle $R^1f_*\Psi^*T_{\Xtwo}$, we take the exact sequence of 
relative tangent bundles associated to the pair 
$(\mathcal C/|K_X|)^{(2)} \subset |K_X| \times \Xtwo$:
$$
0 \to T_{(\mathcal C/|K_X|)^{(2)}/|K_X|} \to 
T_{|K_X| \times \Xtwo/|K_X|}|_{(\mathcal C/|K_X|)^{(2)}} \to 
N_{(\mathcal C/|K_X|)^{(2)} \subset |K_X| \times \Xtwo} \to 0.
$$
Note that $T_{|K_X| \times \Xtwo/|K_X|}|_{(\mathcal C/|K_X|)^{(2)}}
= \W \Psi^*T_\Xtwo$. In addition, we have 
\beq            \label{prop_iii.301}
N_{(\mathcal C/|K_X|)^{(2)} \subset |K_X| \times \Xtwo} = 
\big ( p_{1*}p_2^*\mathcal O_{|K_X| \times X}(\mathcal C) \big )|_{(\mathcal C/|K_X|)^{(2)}}
\eeq
by \cite{AIK}, 
where $p_1: |K_X| \times \mathcal Z_2 \to |K_X| \times \Xtwo$ and 
$p_2: |K_X| \times \mathcal Z_2 \to |K_X| \times X$ are the natural projections. 
Therefore, the above exact sequence becomes
$$
0 \to T_{(\mathcal C/|K_X|)^{(2)}/|K_X|} \to \W \Psi^*T_\Xtwo \to 
\big ( p_{1*}p_2^*\mathcal O_{|K_X| \times X}(\mathcal C) \big )|_{(\mathcal C/|K_X|)^{(2)}} \to 0.
$$
Restricting it to $\mathcal E \subset (\mathcal C/|K_X|)^{(2)}$, we obtain the exact sequence
\beq            \label{prop_iii.4}
0 \to T_{(\mathcal C/|K_X|)^{(2)}/|K_X|}|_{\mathcal E} \to \Psi^*T_\Xtwo \to 
\big ( p_{1*}p_2^*\mathcal O_{|K_X| \times X}(\mathcal C) \big )|_{\mathcal E} \to 0.
\eeq
Let $\mathcal Z_{\mathcal E} = p_1^{-1}(\mathcal E)$. 
Then $p_2 \big ( \mathcal Z_{\mathcal E} \big ) = \mathcal C$.
Put $\w p_1 = p_1|_{\mathcal Z_{\mathcal E}}: \mathcal Z_{\mathcal E} \to \mathcal E$ and 
$\w p_2 = p_2|_{\mathcal Z_{\mathcal E}}: \mathcal Z_{\mathcal E} \to \mathcal C$.
Then the exact sequence \eqref{prop_iii.4} can be rewritten as 
\beq            \label{prop_iii.5}
0 \to T_{(\mathcal C/|K_X|)^{(2)}/|K_X|}|_{\mathcal E} \to \Psi^*T_\Xtwo \to 
\w p_{1*}\w p_2^*\mathcal O_{\mathcal C}(\mathcal C) \to 0.
\eeq
By \eqref{mu0.2.1}, $R^1f_*\big (T_{(\mathcal C/|K_X|)^{(2)}/|K_X|}|_{\mathcal E} \big ) = 0$.
Thus applying $f_*$ to \eqref{prop_iii.5} yields
$$
R^1f_*\Psi^*T_\Xtwo \cong
R^1f_*\big ( \w p_{1*}\w p_2^*\mathcal O_{\mathcal C}(\mathcal C) \big ).
$$
Note that $\w p_2: \mathcal Z_{\mathcal E} \to \mathcal C$ is an isomorphism. 
Via this isomorphism, $\w p_1: \mathcal Z_{\mathcal E} \to \mathcal E$
is identified with the natural double cover $\w p: \mathcal C \to \mathcal E$. So
\beq            \label{prop_iii.6}
R^1f_*\Psi^*T_\Xtwo \cong
R^1f_*\big ( \w p_*\mathcal O_{\mathcal C}(\mathcal C) \big ).
\eeq

Since $K_X^2 = 1$, the complete linear system $|K_X|$ has a unique base point, denoted by $x_0$.
The blowing-up morphism $\w q_2: \W X \to X$ of $X$ at $x_0$ resolves the rational map 
$X \dasharrow |K_X|$, and leads to a morphism $\w q_1: \W X \to |K_X|$.
In addition, $\W X = \mathcal C$, the inclusion $\mathcal C \subset 
|K_X| \times X$ is given by the map $(\w q_1, \w q_2): \W X  \to |K_X| \times X$,
and there exists a commutative diagrams of morphisms:
$$
\begin{array}{ccccc} 
\W X = \mathcal C && \overset {\w p} \longrightarrow && \mathcal E \\
&\searrow^{\w q_1}&&\swarrow_f&\\
&&|K_X|&& 
\end{array} 
$$
From the exact sequence $0 \to T_{\mathcal C} \to T_{|K_X| \times X}|_{\mathcal C} \to 
\mathcal O_{\mathcal C}(\mathcal C) \to 0$, we get
\beq            \label{prop_iii.601}
\mathcal O_{\mathcal C}(\mathcal C) \cong \mathcal O_{\W X}(K_{\W X}) \otimes
\w q_1^*\mathcal O_{|K_X|}(2) \otimes \w q_2^*\mathcal O_X(-K_X)
\cong \mathcal O_{\W X}(E) \otimes \w q_1^*\mathcal O_{|K_X|}(2)
\eeq
where $E \subset \W X$ is the exceptional curve. Combining with \eqref{prop_iii.6}, we obtain
\beq            \label{prop_iii.7}
R^1f_*\Psi^*T_\Xtwo \cong \mathcal O_{|K_X|}(2) \otimes
R^1f_*\big ( \w p_*\mathcal O_{\W X}(E) \big ).
\eeq

Next, we determine the rational ruled surface $\mathcal E$.
Let $\sigma = \w p(E)$. Since $E$ is a section to $\w q_1$, $\sigma$ is a section to $f$.
Let $C$ be a fiber of $\w q_1$. Then $\Gamma \overset {\rm def} = g_2^1(C)$ is 
the fiber of $f$ over the point $\w q_1(C)$, and $K_{\W X} = \w q_2^*K_X + E = C + 2E$.
By the adjunction formula, $K_C = (K_{\W X} + C)|_C = 2(E|_C) = 2(E \cap C)$.
So the double cover $C \to \Gamma$ is ramified at the intersection point $E \cap C$.
It follows that the double cover $\w p$ is ramified along $E$. Thus $\w p^*\sigma = 2E$. 
By the projection formula, $\sigma^2 = (\w p^*\sigma)^2/2 = 2E^2 = -2$. Thus, 
$\mathcal E$ is the Hirzebruch surface $\mathbb F_2 = 
{\mathbb P}\big ( \mathcal O_{|K_X|} \oplus \mathcal O_{|K_X|}(-2) \big )$
with $\mathcal O_{\mathcal E}(1) = \mathcal O_{\mathcal E}(\sigma)$.
It follows that $K_{\mathcal E} = -2\sigma - 4\Gamma$.

Let $B \subset \mathcal E$ be the branch locus of the double cover $\w p$,
and $\W B = \w p^{-1}(B)$. Then $B$ and $\W B$ are smooth, $\sigma \subset B$, 
$E \subset \W B$, and $B = 2L$ for some divisor $L$ on $\mathcal E$. Also,
$$
C + 2E = K_{\W X} = \w p^*(K_{\mathcal E}) + \W B = \w p^*(-2\sigma - 4\Gamma) + \W B
= -4E - 4C + \W B.
$$
So $\W B = 6E + 5C$, and $B = \w p_*\W B = 6\sigma + 10 \Gamma$ (thus $B$ is the disjoint union 
of $\sigma$ and a smooth curve in $|5\sigma + 10 \Gamma|$).
Since $\w p_*\mathcal O_{\W X} = \mathcal O_X \oplus \mathcal O_X(-L)$,
\begin{eqnarray}    \label{prop_iii.8}
   c_1\big ( \w p_*\mathcal O_{\W X}(E) \big ) 
&=&{\rm ch}_1\big ( \w p_!(\mathcal O_{\W X}(E)) \big ) 
   = \big \{ \w p_*({\rm ch}( \mathcal O_{\W X}(E)) \cdot {\rm td}(T_{\w p})) \big \}_1 
                \nonumber  \\
&=&\w p_*E + c_1\big ( \w p_*\mathcal O_{\W X} \big )
   = \sigma + (-L) = -2\sigma - 5\Gamma
\end{eqnarray}
by the Grothendieck-Riemann-Roch Theorem, where $T_{\w p}$ is the relative tangent sheaf of $\w p$.
Since $h^0(\mathcal E, \w p_*\mathcal O_{\W X}(E)) = h^0(\W X, \mathcal O_{\W X}(E)) = 1$,
there exists an injection $\mathcal O_{\mathcal E} \to \w p_*\mathcal O_{\W X}(E)$
which in turn induces an injection $\mathcal O_{\mathcal E}(a\sigma) \to 
\w p_*\mathcal O_{\W X}(E)$ with $a \ge 0$ and with torsion-free quotient 
$\w p_*\mathcal O_{\W X}(E)/\mathcal O_{\mathcal E}(a\sigma)$. So we have an exact sequence
\beq            \label{prop_iii.9}
0 \to \mathcal O_{\mathcal E}(a\sigma) \to \w p_*\mathcal O_{\W X}(E) \to 
\mathcal O_{\mathcal E}((-a-2)\sigma - 5\Gamma) \otimes I_\eta \to 0
\eeq
where $\eta$ is a $0$-cycle on $\mathcal E$. 
By \eqref{mu0.5}, \eqref{prop_iii.301} and \eqref{prop_iii.601}, we get
\begin{eqnarray*}
   \mathcal O_\Gamma \oplus \mathcal O_\Gamma(-2) 
&=&N_{C^{(2)} \subset \Xtwo}|_\Gamma 
   = \big ( N_{(\mathcal C/|K_X|)^{(2)} \subset |K_X| \times \Xtwo} \big )|_\Gamma  \\
&=&\big ( p_{1*}p_2^*\mathcal O_{|K_X| \times X}(\mathcal C) \big )|_\Gamma 
   = \big ( \w p_{1*}\w p_2^*\mathcal O_{\mathcal C}(\mathcal C) \big )|_\Gamma    \\
&=&\Big ( \w p_*\big ( \mathcal O_{\W X}(E) \otimes \w q_1^*\mathcal O_{|K_X|}(2) 
       \big ) \Big )|_\Gamma
   = \big ( \w p_*\mathcal O_{\W X}(E) \big )|_\Gamma.
\end{eqnarray*}
Since this holds for every fiber $\Gamma$ of $f$, we conclude from \eqref{prop_iii.9} that 
$a = 0$ and $\eta = \emptyset$. So the exact sequence \eqref{prop_iii.9} is simplified to
$$
0 \to \mathcal O_{\mathcal E} \to \w p_*\mathcal O_{\W X}(E) \to 
\mathcal O_{\mathcal E}(-2\sigma - 5\Gamma) \to 0.
$$
Applying the functor $f_*$ to the above exact sequence yields 
$$
R^1f_*\big ( \w p_*\mathcal O_{\W X}(E) \big ) 
\cong R^1f_*\mathcal O_{\mathcal E}(-2\sigma - 5\Gamma)
\cong \mathcal O_{|K_X|}(-5) \otimes R^1f_*\mathcal O_{\mathcal E}(-2\sigma). 
$$
Since $\mathcal E = {\mathbb P}\big ( \mathcal O_{|K_X|} \oplus \mathcal O_{|K_X|}(-2) \big )$
with $\mathcal O_{\mathcal E}(1) = \mathcal O_{\mathcal E}(\sigma)$ and 
$f_*\mathcal O_{\mathcal E} = \mathcal O_{|K_X|}$,
$$
R^1f_*\big ( \w p_*\mathcal O_{\W X}(E) \big )
\cong \mathcal O_{|K_X|}(-5) \otimes \Big ( (f_*\mathcal O_{\mathcal E})^\vee 
      \otimes \mathcal O_{|K_X|}(-2)^\vee \Big )
\cong \mathcal O_{|K_X|}(-3).
$$
By \eqref{prop_iii.7}, $R^1f_*\Psi^*T_\Xtwo \cong \mathcal O_{|K_X|}(-1)$.
Finally, by \eqref{prop_iii.3}, we obtain 
\begin{equation}
\langle 1 \rangle_{0, \,\, \beta_{K_X} - 3\beta_2}^{\Xtwo}
= \deg \, [\Mbar]^{\text{\rm vir}} = \deg \, c_1 \big (R^1f_*\Psi^*T_\Xtwo \big ) = -1.
\tag*{$\qed$}
\end{equation}

\end{document}